\documentclass{imsart}

\RequirePackage[OT1]{fontenc}
\usepackage{a4wide}

\usepackage[ruled,linesnumbered]{algorithm2e}
\usepackage{xr-hyper}
\usepackage{hyperref}

\usepackage{amsmath,amsthm,amssymb,mathrsfs}
\usepackage{stmaryrd}
\usepackage{nicefrac}
\usepackage{mathtools} 
\usepackage[utf8]{inputenc}
\usepackage{amsfonts,dsfont}

\usepackage{soul,graphicx}
\graphicspath{
  {fig/}
}
\usepackage{natbib}
  \usepackage{pgf, tikz}
  \usetikzlibrary{fit,patterns,decorations.pathreplacing}

\startlocaldefs

\usepackage{amsthm}
\newtheorem{theorem}{Theorem}[section]
\newtheorem{proposition}[theorem]{Proposition}
\newtheorem{corollary}[theorem]{Corollary}
\newtheorem{lemma}[theorem]{Lemma}

\theoremstyle{remark}
\newtheorem{remark}[theorem]{Remark}
\newtheorem{example}[theorem]{Example}
\theoremstyle{definition}
\newtheorem{definition}[theorem]{Definition}

\usepackage[ruled]{algorithm2e}
\newcommand{\VBonf}{V_{\mbox{\tiny Bonf}}}
\newcommand{\VSimes}{V_{\mbox{\tiny Simes}}}
\newcommand{\VDKW}{V_{\mbox{\tiny DKW}}}
\newcommand{\Vhyb}{V_{\mbox{\tiny hybrid}}}
\newcommand{\Rfam}{\mathfrak{R}}

\newcommand{\con}{\texttt{con}}
\newcommand{\len}{\texttt{len}}

\renewcommand{\P}{\mathbb{P}}
\newcommand{\E}{\mathbb{E}}

\newcommand{\telque}{\,:\,}

\newcommand{\mtc}{\mathcal}

\newcommand{\wt}[1]{{\widetilde{#1}}}

\newcommand{\ol}[1]{\overline{#1}}
\newcommand{\ind}[1]{{\mathds{1}\left\{#1\right\}}}

\newcommand{\paren}[1]{\left(#1\right)}

\newcommand{\set}[1]{\left\{#1\right\}}
\newcommand{\abs}[1]{\left| #1 \right|}

\newcommand{\cH}{{\mtc{H}}}

\newcommand{\cP}{{\mtc{P}}}

\newcommand{\cN}{{\mtc{N}}}

\renewcommand{\l}{\ell}

\newcounter{nbdrafts}
\makeatletter
\newcommand{\checknbdrafts}{
\ifnum \thenbdrafts > 0
\@latex@warning@no@line{**********************************************************************}
\@latex@warning@no@line{* The document contains \thenbdrafts \space draft note(s)}
\@latex@warning@no@line{**********************************************************************}
\fi}

\makeatother

\newcommand{\Nm}{\mathbb{N}_m}

\newcommand{\Vtilde}{\wt{V}_{\Rfam}}
\newcommand{\Vtildeq}{\wt{V}^q_{\Rfam}}
\newcommand{\Vbar}{\ol{V}_{\Rfam}}
\newcommand{\setK}{\mathcal{K}}
\newcommand{\setC}{\mathcal{C}}

\endlocaldefs

\newcommand{\N}[1]{\mathbb{N}_{#1}}
\newcommand{\comp}[1]{{#1}^{\mathsf{c}}} 
\newcommand{\Pro}[1]{\mathbb{P}\left(#1\right)}
\newcommand{\Esp}[1]{\mathbb{E}\left[ #1 \right]}
\newcommand{\ti}{{i'}}
\newcommand{\tj}{{j'}}
\newcommand{\tk}{{k'}}

\newcommand{\h}{\phi}
\newcommand{\Rf}[2]{P_{#1:#2} }

\begin{document}

\begin{frontmatter}

\title{Post hoc false positive control for spatially structured hypotheses}
\runtitle{Post hoc false positive control}
\begin{aug}
\author{\fnms{Guillermo} \snm{Durand}
\ead[label=e1]{guillermo.durand@upmc.fr}}
\address{
Sorbonne Universit\'e (Universit\'e Pierre et Marie Curie), LPSM,\\ 4, Place Jussieu, 75252 Paris cedex 05, France\\
\printead{e1}\\
}


\author{\fnms{Gilles} \snm{Blanchard}
\ead[label=e2]{gilles.blanchard@math.uni-potsdam.de}}
\address{
Universit\"at Potsdam, Institut f\"ur Mathematik\\
Karl-Liebknecht-Stra{\ss}e 24-25 14476 Potsdam, Germany\\
\printead{e2}\\
}

\author{\fnms{Pierre} \snm{Neuvial}\ead[label=e3]{pierre.neuvial@math.univ-toulouse.fr}}
\address{
Institut de Math\'ematiques de Toulouse; \\
UMR 5219, Universit\'e de Toulouse, CNRS\\
UPS IMT, F-31062 Toulouse Cedex 9, France\\
\printead{e3}\\
}
\author{\fnms{Etienne} \snm{Roquain}
\ead[label=e4]{etienne.roquain@upmc.fr}}
\address{
Sorbonne Universit\'e (Universit\'e Pierre et Marie Curie), LPSM,\\ 4, Place Jussieu, 75252 Paris cedex 05, France\\
\printead{e4}\\
}

\end{aug}
\begin{abstract}
In a high dimensional multiple testing framework, we present new confidence bounds on the false positives contained in subsets $S$ of selected null hypotheses. The coverage probability holds simultaneously over all 
 subsets $S$, which means that the obtained confidence bounds are post hoc. Therefore,  $S$ can be chosen arbitrarily, possibly by using the data set several times.
We focus in this paper specifically on the case where the null hypotheses are spatially structured.
Our method is based on recent advances in post hoc inference and particularly on 
the general methodology of \citet{blanchard2017post}; we build confidence bounds for some pre-specified forest-structured subsets $\{R_k,k\in \mathcal{K}\}$, called the reference family, and then we deduce a bound for any subset $S$ by interpolation.
The proposed bounds are shown to improve substantially previous ones when 
the signal is locally structured. Our findings are supported both by theoretical results and numerical experiments.
Moreover, we show that our bound can be obtained by a low-complexity algorithm, which makes our approach completely operational for a practical use. The proposed bounds are implemented in the  open-source \textsf{R} package \texttt{sansSouci}\footnote{available from \url{https://github.com/pneuvial/sanssouci}.}.
\end{abstract}



\end{frontmatter}


\section{Introduction}

\subsection{Background}

Modern statistical data analysis often involves asking many questions of interest simultaneously, possibly using the data repeatedly, as long as the user feels that this could provide additional information. To avoid selection bias due to
various forms of data snooping, specific strategies can be proposed to take
into account the procedure as whole, and be investigated as to the statistical
guarantees they provide. This problem is often referred to as selective inference, a long standing research field, with a recent renewal of interest. An historical reference is the work of \citet{scheffe1953method} (see also \citealp[p. 69]{scheffe1959analysis}), which is to our knowledge the earliest work proposing simultaneous selective 
 inference.
 In the context of linear regression,  \citet{berk2013valid} proposed an improvement of this Scheff\'e protection by defining a less conservative correction term (the so-called PoSI constant), see also \citet{bachoc2014valid,bachoc2018post} for recent developments on this issue.

Other strategies perform inference on the observed selection set only, either by a false coverage rate control \citep{benjamini2005false,benjamini2014selective}
or by a controlling a criterion conditional to a specific initial selection step, see the series of works \citet{fithian2014optimal,taylor2015statistical,tibshirani2016exact,choi2017selecting,taylor2018post}. In other studies, the selection step is based on sample splitting, see \citet{cox1975note,buhlmann2014high, dezeure2015high}, which is another way to tackle selective inference by explicitly avoiding data reuse.


%


We follow in this paper the aim of establishing confidence bounds on the number of false positives in a multiple testing framework,
simultaneously over all possible set of selected hypotheses.
If we observe a random variable $X\sim P$, $P$ belonging to some model $\mtc{P}$, for which $m$ null hypotheses $H_{0,i} \subset \mtc{P} $, $i \in \Nm = \set{1,\ldots,m}$ are
under investigation for $P$, the aim is to build
a function ${V}(X,\cdot) : S \subset \Nm \mapsto {V}(X,S)$ (denoted by ${V}(S)$ for short) satisfying 
\begin{equation}\label{aim}
\forall P\in\mtc{P},\, \qquad \P_{X\sim P}\Big(\forall S \subset 
\Nm ,\:
|S\cap \cH_0(P)| \leq {V}(S) \Big)\geq 1-\alpha,
\end{equation}
where $\cH_0(P)=\{i \in \Nm \telque P \mbox{ satisfies } H_{0,i}\}$ is the set of true null hypotheses.
The bound $V(\cdot)$ will be referred to as a post hoc bound throughout this manuscript.

The problem of constructing post hoc bounds has been first tackled specifically in the case where the selection sets $S$ are of the form of $p$-value level sets: $\{i\::\:p_i(X)\leq t\}$, $t\in[0,1]$, where each $p_i(X)$ is a $p$-value for the null hypothesis $H_{0,i} $, $1\leq i\leq m$. The resulting bounds are often referred to as confidence envelopes, see \citet{genovese2004stochastic,meinshausen2006false}.
Later, \citet{genovese2006exceedance} and \citet{goeman2011multiple}  proposed to extend this approach to arbitrary subsets $S$, by using a methodology based on performing $2^m-1$ local tests (one for each intersection hypothesis), with a possible complexity reduction by using shortcuts.
In particular, the approach of \citet{goeman2011multiple} extensively relies on the closed testing principle, which was 
introduced by \cite{marcus1976closed}. 
This approach has been further extended in \citet{meijer2015multiple,meijer2015region} by using the sequential rejection principle of \citet{goeman2010sequential}. This allows to incorporate structural informations  into the post hoc bound. 
In particular, the method in \citet{meijer2015region}, whose goal inspired the present work, deals with geometrically structured null hypotheses, along space or time and shows that incorporating such an external information can substantially improve the detection of signal and thus can increase the accuracy of the resulting post hoc bound.

More recently, \citet{blanchard2017post} (BNR below) have proposed a flexible methodology
that adjusts the complexity of the bound by way of a reference family:
the post hoc bound is based on a family
  $\Rfam=((R_k(X),\zeta_k(X))_{k\in\setK}$ ($R_k,\zeta_k$ for short), with $R_k\subset \Nm$ (and $R_k\neq R_{\tk}$ if $k\neq\tk$), $\zeta_k\in \mathbb{N}$,
that satisfies the following joint error rate (JER) control:
  \begin{equation}
\forall P \in \cP, \qquad
 \qquad \P_{X\sim P}\Big(\forall k  \in \setK ,\:
|R_k\cap \cH_0(P)| \leq \zeta_k \Big)\geq 1-\alpha,
\label{eq:jercontrol}
\end{equation}
An important difference between \eqref{aim} and \eqref{eq:jercontrol} is that $S$ in \eqref{aim} is let arbitrary and typically chosen by the user, whereas $R_k,\zeta_k$ in \eqref{eq:jercontrol} is part of the  methodology and is chosen by the statistician to make \eqref{eq:jercontrol} hold.
Once the reference family is fixed, a post hoc bound is obtained from \eqref{eq:jercontrol} simply by interpolation, by exploiting the constraints that the event in \eqref{eq:jercontrol} imposes to the unknown set $\cH_0(P)$, namely that it is a subset $A$ with the property "$\forall k  \in \setK ,\:
|R_k\cap A| \leq \zeta_k$":
\begin{equation}\label{eq:optbound1}
  V^*_\Rfam(S) = \max\big\{ \abs{S\cap A}, A \subset \Nm, \forall k  \in \setK ,\:
|R_k\cap A| \leq \zeta_k \big\}, \:\:\: S\subset \Nm\,.
\end{equation}
Hence, if \eqref{eq:jercontrol} holds, then $ V=V^*_\Rfam$ satisfies \eqref{aim}. This post-hoc bound will be referred to as the {\it optimal bound} (relative
to a given reference family).

\subsection{Contributions of the paper}

In this paper, we propose new post hoc bounds that 
incorporate the specific spatial structure of the null hypotheses. While this aim is similar in spirit to 
\citet{meijer2015region}, our method is markedly different, as it relies on the general strategy laid down by BNR, with a specifically structured reference family $R_k, k\in \setK$ (see Section~\ref{sec:compaGoeman} for a comparison between our approach and the one of \citealp{meijer2015region}). In addition, 
the way the method is built here is different than the one proposed in Section~3-6 of BNR:
 the main focus in BNR is the case of (random) reference sets $R_k=R_k(X)$ that are designed in order to satisfy \eqref{eq:jercontrol} with  $\zeta_k=k-1$ (thus corresponding to a ``joint $k$-family-wise error rate''). By contrast, in the present work the reference sets $R_k$ are fixed in advance, and the (random) bounds on the number false positives $\zeta_k=\zeta_k(X)$ are designed to satisfy the constraint \eqref{eq:jercontrol}. 
The rationale behind this approach is that the reference sets $R_k$ can be chosen arbitrarily by the statistician, so that it can accommodate any pre-specified structure (reflecting some prior knowledge on the considered problem). Since we are interested in structured signal, we focus on a reference family enjoying a forest structure, meaning that two reference sets are either disjoint or nested. 

The second ingredient of our method is the local bounds $\zeta_k(X)$, that should estimate $|R_k\cap \cH_0(P)|$ with a suitable deviation term. While any deviation inequality can be used, we have chosen to focus on the DKW inequality \citep{dvoretzky1956asymptotic}, that has the advantage to be sub-Gaussian. Hence, the uniformity over the range $k\in \setK$ can be obtained by a simple union bound without being too conservative.

Let us mention that using the DKW inequality to obtain a confidence bound for the proportion of null hypotheses is not new, see \citet{genovese2004stochastic} (Equation~(16) therein), \citet{meinshausen2006false}, and \citet{farcomeni2011conservative}. While our bound is a uniform improvement of the existing version (see Remark~\ref{rem:comparisonpreviouspi0} below for more details),
our main innovation is to use the DKW bound in a local manner and to appropriately combine these local bounds to derive an overall post hoc bound. The improvement can be substantial, as illustrated in our numerical experiments.



The paper is organized as follows: precise setup and notation are introduced in Section~\ref{sec:prel}. For any reference family with a forest structure, the optimal post hoc bound is computed in Section~\ref{sec:Forest}.
The calibration of the local bounds $\zeta_k$ and of the overall reference family is done in Section~\ref{sec:calibration}. This section also includes a theoretical comparison with previous methods, which quantifies formally the amplitude of the improvement induced by the new method. The latter is supported by numerical experiments in Section~\ref{sec:num}, where a hybrid approach is also introduced to mimic the best between the new approach and the existing Simes bound (the latter being defined in \eqref{equ:SimesV} below). A discussion is given in Section~\ref{sec:dis} and the proofs are provided in Section~\ref{sec:proof}. Additional technical details are postponed to Appendices~\ref{app:sec:aux} and \ref{app:sec:leaves}.


\section{Preliminaries}\label{sec:prel}

\subsection{Assumptions}

We focus on the common situation where a test statistic $T_i(X)$ is available for each null hypothesis $H_{0,i}$. For $i \in \Nm$, each statistic $T_i(X)$ is transformed into a $p$-value $p_i(X)$, satisfying the following assumptions:
 \begin{align}
&\forall i \in\mathcal{H}_0, \:\:\forall t\in[0,1], \:\:\Pro{p_i(X)\leq t} \leq t;
\tag{Superunif}
\label{eq:superunif}\\
&\{p_i(X)\}_{i\in\mathcal{H}_0} \text{ is a family of independent $p$-values and is independent of } \{p_i(X)\}_{i\in\mathcal{H}_1}.
\tag{Indep}
\label{eq:indep}
\end{align}
Extending our results to the case where \eqref{eq:indep} fails is possible, see the discussion in Section~\ref{sec:dis}.

\subsection{Classical post hoc bounds}

As argued in BNR, computing the optimal post hoc bound \eqref{eq:optbound1} relative
to a given reference family $(R_k,\zeta_k)_{k \in \mtc{K}}$ can be  NP-hard, and simpler, more conservative versions can be provided, that is, bounds $V$ such that for all $S\subset \Nm$, $V^*_{\Rfam}(R)\leq V(R)$.
A simple upper-bound for $V^*_{\Rfam}$ is given by
\begin{align}
\Vbar(S)&= |S|\wedge \min_{k\in\setK}\left\{\zeta_{k}+ |S\setminus R_k|  \right\} , \:\:\: S\subset \Nm\,.
\label{eq:Vbar}
\end{align}
It is straightforward to check that
\begin{align}
V^*_{\Rfam}(S)\leq \Vbar(S), \:\:\: S\subset \Nm.
\label{eq:Vbarconservative}
\end{align}
While this inequality is strict in general, BNR established that it is an equality if the reference family is nested, that is,
\begin{equation}\label{equ:nested}\tag{Nested}
\mbox{$\setK=\{1,\dots,K\}$ and $R_k\subset R_{k+1}$ for $1\leq k \leq K-1$. }
\end{equation}
Condition \eqref{equ:nested} is mild when the sequence $\zeta_k$ is nondecreasing, e.g., $\zeta_k=k-1$.

A consequence of \eqref{eq:Vbarconservative} is that $\Vbar$ is a post hoc bound in the sense of \eqref{aim} as soon as the reference family $\Rfam$ is such that \eqref{eq:jercontrol} holds.
A simple union bound under \eqref{eq:superunif} yields that \eqref{eq:jercontrol} holds with $\Rfam=\{(R_1,\zeta_1)\}$, $R_1=\{i\in\Nm\::\: p_i \leq \alpha/m\}$, $\zeta_1=0$. This leads to the Bonferroni post hoc  bound
\begin{align}\label{equ:BonfV}
\VBonf(S) = \sum_{i\in S}\ind{p_i(X)>\alpha/m}, \:\:\:\:\:S\subset \Nm.
\end{align}
The more subtle Simes inequality \citep{simes1986improved}, valid under \eqref{eq:superunif}--\eqref{eq:indep}, ensures that  \eqref{eq:jercontrol} holds with $\Rfam=\{(R_k,\zeta_k),1\leq k \leq m\}$, $R_k=\{i\in\Nm\::\: p_i \leq \alpha k/m\}$, $\zeta_k=k-1$.
This leads to the Simes post hoc bound
\begin{align}\label{equ:SimesV}
\VSimes(S) = \min_{1\leq k \leq m}\left\{\sum_{i\in S}\ind{p_i(X)>\alpha k/m} + k-1\right\}, \:\:\:\:\:S\subset \Nm.
\end{align}
As noted in BNR, this bound is identical to post hoc bound of \citet{goeman2011multiple}, which will be used as a benchmark in this paper.

\subsection{Improved interpolation bound}

When the sequence $\zeta_k$ is not nondecreasing, inequality \eqref{eq:Vbarconservative} can be far too conservative. 
We introduce the following extension: for a reference family $\Rfam=(R_k(X),\zeta_k(X))_{k\in\setK}$ of cardinal $K=|\setK|$,
\begin{align}
\Vtildeq(S)&=\min_{Q\subset \setK, |Q|\leq q} \Bigg( \sum_{k\in Q} \zeta_{k}\wedge|S\cap R_{k}| +\bigg|S\setminus\bigcup_{k\in Q} R_{k}\bigg|  \Bigg) , \:\:\: 1\leq q\leq K , \:\:\: S\subset \Nm\,;
\label{eq:Vtildeq}\\
\Vtilde(S)&= \Vtilde^{K}(S), \:\:\: S\subset \Nm\,.
\label{eq:VtildeK}
\end{align}
Obviously, we have $\Vtilde^1=\Vbar$ and $\Vtildeq$ is non-increasing in $q$.
The following result shows that these bounds are all conservative versions of $V^*_{\Rfam}$.
\begin{lemma}\label{lem:Vtilde}
For any reference family $\Rfam$,  we have
\begin{equation}\label{equ:Vtildeconservative}
V^*_{\Rfam}(S)\leq \Vtilde(S)\leq \Vtildeq(S)\leq \Vbar(S), \:\:\: 1\leq q\leq K,\:\:\:S\subset \Nm.
\end{equation}
 In particular, if $\Rfam$ is such that \eqref{eq:jercontrol} holds, then $\Vtilde$ is a post hoc bound in the sense of \eqref{aim}.
\end{lemma}

Lemma~\ref{lem:Vtilde} is proved in Section~\ref{proof:firstlemma}.
The inequality $V^*_{\Rfam}(S)\leq \Vtilde(S)$ in \eqref{equ:Vtildeconservative} is strict in general, see Example~\ref{contre_no}.
As we will show in the next section, this relation is nevertheless an equality when $\Rfam$ has a specific forest structure, which makes $\Vtilde$ a particularly interesting bound.

\begin{example}
\label{contre_no}
Let $m=4$, 
$K=3$, $R_1=\{1,2,4\}$, $R_2=\{2,3,4\}$, $R_3=\{1,3,4\}$. Consider  the event where $\zeta_1(X)=\zeta_2(X)=\zeta_3(X)=1$. For $S=\N{4}$, we easily check that $V^*_{\Rfam}(S)=1$ and $\Vtilde(S)=2$.
\end{example}

\section{Post hoc bound for forest structured reference family}\label{sec:Forest}


\subsection{Forest structure}\label{sec:forest}


\begin{definition}
A reference family $\Rfam=(R_k,\zeta_k)_{k\in\setK}$ is said to have a forest structure if following property is satisfied:
\begin{equation}
\forall k,\tk\in\setK , \:\:R_k \cap R_\tk \in \{ R_k,  R_\tk , \varnothing \},
\label{eq:Forest}\tag{Forest}
\end{equation}
that is, two elements of $\{R_k\}_{k\in\setK}$ are either disjoint or nested.
\label{def:Forest}
\end{definition}

The forest structure is general enough to cover a wide range of different situations, as for instance the disjoint case
\begin{equation}
\forall k,\tk\in\setK , \:\:  k\neq\tk \Rightarrow R_k \cap R_\tk=\varnothing.
\label{equ:disjoint}\tag{Disjoint}
\end{equation}
and the nested case  \eqref{equ:nested}. 
In general, if each $R_k$ is considered as a node and if an oriented edge $R_k \leftarrow R_{\tk}$ is depicted between two different sets $R_k$ and $R_{\tk}$  if and only if $R_k \subset R_{k'}$ and there is no $R_{k''}$ such that $R_k \subsetneq R_{k''}\subsetneq R_{k'}$; the obtained graph correspond to a (directed) forest in the classical graph theory sense, see e.g. \citet{kolaczyk2009statistical}. An illustration is given in  Figure~\ref{fig:graph}.
The positions of the nodes in this picture rely on the depth of $\Rfam$, which can be defined as the function
\begin{equation}
\h \: : \: \left\{
\begin{array}{l  c l  }
 \setK & \to & \mathbb{N}^*\\
k & \mapsto & 1 + \left| \{\tk\in\setK: R_\tk\supsetneq R_k \} \right|   .
\end{array}
\right.
\label{eq:depth}
\end{equation}
For instance, under \eqref{equ:disjoint}, $\phi(k)=1$ for all $k\in \setK$, while under \eqref{equ:nested}, $ \h(k)=K+1-k$ for all $1\leq k\leq K$.


\begin{example}\label{ex:base}
Let $m=25$, $R_1 = \{1, \dotsc , 20 \}$, $R_2  =  \{1, 2  \}$, $R_3   =   \{3 , \dotsc , 10 \}$, $R_4  =    \{11, \dotsc , 20 \}$, $R_5 =  \{5, \dotsc , 10 \}$, $R_6   =     \{11, \dotsc , 16 \}$, $R_7  =   \{17, \dotsc ,20  \}$, $R_8=\{21,22\}$, $R_9 = \{22\}$. Then the corresponding reference family $\Rfam=(R_k,\zeta_k)_{1\leq k\leq 9}$ satisfies \eqref{eq:Forest}.
The sets $R_1$, $R_8$ are of depth~$1$; the sets $R_2,R_3,R_4,R_9$ are of depth~$2$; the sets $R_5,R_6,R_7$ are of depth~$3$.
\end{example}

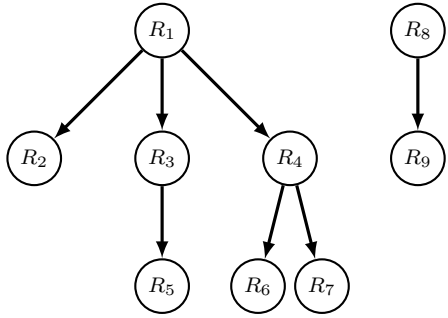
\begin{figure}[h!]
\begin{center}
\begin{tikzpicture}[scale=0.85]
 \tikzstyle{quadri}=[circle,draw,text=black, thick]
 \tikzstyle{estun}=[->,>=latex,very thick]
 \node[quadri] (R1) at (0,3) {$R_1$};
 \node[quadri] (R2) at (-2,1) {$R_2$};
 \node[quadri] (R3) at (0,1) {$R_3$};
 \node[quadri] (R4) at (2,1) {$R_4$};
  \node[quadri] (R5) at (0,-1) {$R_5$};
  \node[quadri] (R6) at (1.5,-1) {$R_6$};
  \node[quadri] (R7) at (2.5,-1) {$R_7$};
  \node[quadri] (R8) at (4,3) {$R_8$};
  \node[quadri] (R9) at (4,1) {$R_9$};
  \draw[estun] (R1)--(R2);
  \draw[estun] (R1)--(R3);
  \draw[estun] (R1)--(R4);
  \draw[estun] (R3)--(R5);
    \draw[estun] (R4)--(R6);
  \draw[estun] (R4)--(R7);
 \draw[estun] (R8)--(R9);
\end{tikzpicture}
\end{center}
\caption[Graph of the reference family of Example~\ref{ex:base}]{
Graph corresponding to the reference family given in Example~\ref{ex:base}.
\label{fig:graph}}
\end{figure}
A useful characterization of a forest-structure reference family is given in the next lemma.
\begin{lemma}\label{lm:partition}
For any reference family $\Rfam=(R_k,\zeta_k)_{k\in\setK}$ having the structure \eqref{eq:Forest}, there exists a partition $(P_n)_{1\leq n \leq N}$ of $\Nm$ such that
for each $k\in \setK$, there exists some $(i,j)$ with $1\leq i\leq j\leq N$ and $R_k=\Rf{i}{j}$, where we denote
\begin{equation}
\Rf{i}{j}=\bigcup_{i\leq n \leq j}P_n,\:\:\:\:\:\:\:1\leq i\leq j\leq N.
\label{eq:Rij}
\end{equation}
Conversely, for some partition $(P_n)_{1\leq n \leq N}$ of $\Nm$, consider any reference family of the form $\Rfam=(\Rf{i}{j},\zeta_{i,j})_{(i,j)\in\setC}$ with $\setC\subset \{(i,j)\in \mathbb{N}_N^2 \::\: i\leq j\}$ such that for $(i,j), (\ti,\tj)\in\setC$, we have
$$\mbox{
$\llbracket i,j \rrbracket\cap \llbracket \ti,\tj \rrbracket=\varnothing$;
or
$\llbracket i,j \rrbracket\subset \llbracket \ti,\tj\rrbracket$;
or
$\llbracket \ti,\tj \rrbracket \subset \llbracket i,j\rrbracket$
},$$
where $\llbracket i,j \rrbracket$ denotes the set of all integers between $i$ and $j$.
%
%
Then $\Rfam$ has the structure \eqref{eq:Forest}.
\end{lemma}

For the ease of notation, the set $\setC$ will be identified to $\setK$ throughout the paper, which leads to the following slight  abuse: denoting indifferently $k\in\setK$ or $(i,j)\in\setK$, and
\begin{equation}\label{localstruct}
\Rfam=(R_k,\zeta_k)_{k\in\setK} \:\:\mbox{ or }\:\:\Rfam=(\Rf{i}{j},\zeta_{i,j})_{(i,j)\in\setK}.
\end{equation}

We call ``atoms'' the elements of the underlying partition $(P_n)_{1\leq n \leq N}$ because they have the thinnest granularity in the structure and because any subset $R_k$ of the family can be expressed as a combination of these atoms. Note however that this partition is not unique. A simple algorithm to compute $(P_n)_n$ and the proof of Lemma~\ref{lm:partition} are provided in Appendix~\ref{app:sec:leaves}.
An example of such a partition is  given in Example~\ref{ex:base2} and Figure~\ref{fig:graphwithpartition}.

\begin{example}\label{ex:base2}
For the reference family given in Example~\ref{ex:base}, a partition as in Lemma~\ref{lm:partition} is given by
$P_1 =R_2$, $P_2  =   R_3\setminus R_5$, $P_3  =   R_5$, $P_4=R_6$, $P_5=R_7$, $P_6=R_8\setminus R_9$, $P_7=R_9$, $P_8=\Nm \setminus \{R_1 \cup R_8 \}$.
%
%
\end{example}

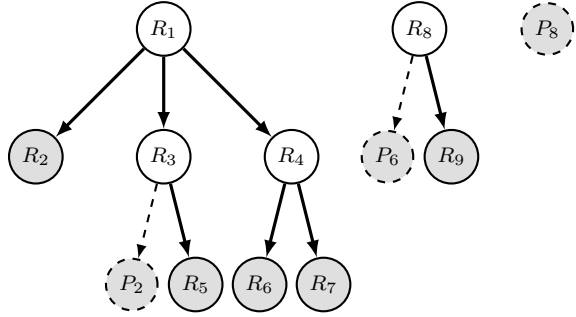
\begin{figure}[h!]
\begin{center}
\begin{tikzpicture}[scale=0.85]
 \tikzstyle{quadri}=[circle,draw,text=black,thick]
 \tikzstyle{estun}=[->,>=latex,very thick]
 \node[quadri] (R1) at (0,3) {$R_1$};
 \node[quadri, fill=gray!25] (R2) at (-2,1) {$R_2$};
 \node[quadri] (R3) at (0,1) {$R_3$};
 \node[quadri] (R4) at (2,1) {$R_4$};
  \node[quadri, dashed,fill=gray!25] (P2) at (-0.5,-1) {$P_2$};
  \node[quadri,fill=gray!25] (R5) at (0.5,-1) {$R_5$};
  \node[quadri,fill=gray!25] (R6) at (1.5,-1) {$R_6$};
  \node[quadri,fill=gray!25] (R7) at (2.5,-1) {$R_7$};
  \node[quadri] (R8) at (4,3) {$R_8$};
  \node[quadri,dashed,fill=gray!25] (P6) at (3.5,1) {$P_6$};
  \node[quadri,fill=gray!25] (R9) at (4.5,1) {$R_9$};
  \node[quadri,dashed,fill=gray!25] (P8) at (6,3) {$P_8$};
  \draw[estun] (R1)--(R2);
  \draw[estun] (R1)--(R3);
  \draw[estun] (R1)--(R4);
  \draw[estun] (R3)--(R5);
  \draw[estun,dashed,thick] (R3)--(P2);
    \draw[estun] (R4)--(R6);
  \draw[estun] (R4)--(R7);
 \draw[estun] (R8)--(R9);
  \draw[estun,dashed,thick] (R8)--(P6);
\end{tikzpicture}
\end{center}
\caption[Graph of the reference family of Example~\ref{ex:base} with atoms]{
Graph corresponding to the reference family given by Example~\ref{ex:base}, with the associated partition (atoms) $\{P_n,1\leq n \leq N\}$, displayed by light gray nodes and given in Example~\ref{ex:base2}. The nodes that correspond to atoms that are not in the reference family are depicted with a dashed circle.
\label{fig:graphwithpartition}}
\end{figure}

An important particular case in our analysis is the case where the forest structure includes all atoms, that is
\begin{equation}
\forall n\in\{1,\dotsc,N\}, \:\:P_n\in\{ R_k, k\in\setK\} .
\tag{All-atoms}
\label{eq:atom}
\end{equation}
When \eqref{eq:atom} does not hold (as in Example~\ref{ex:base2}), we can impose this condition by adding them to the structure, building in this way the completed reference family:

\begin{definition}\label{def:complet}
Consider any reference family $\Rfam=(\Rf{i}{j},\zeta_{i,j})_{(i,j)\in\setK}$ satisfying \eqref{eq:Forest} and associated to atoms $(P_n)_{1\leq n \leq N}$ by \eqref{localstruct}.
Let $\setK^+=\{(i,i), 1\leq i\leq N : (i,i)\not\in\setK  \}$, $\zeta_{i,i}=|\Rf{i}{i}|=|P_i|$ for all $(i,i)\in\setK^+$, and $\setK^\oplus=\setK\cup\setK^+$. 
Then the completed version of $\Rfam$ is given by
$
\Rfam^\oplus=(\Rf{i}{j},\zeta_{i,j} )_{(i,j)\in\setK^\oplus}.
$\end{definition}
For the reference family $\Rfam$ given by Example~\ref{ex:base}, the completed version $\Rfam^\oplus$ is depicted in Figure~\ref{fig:completedgraph}.

\begin{figure}[h!]
\begin{center}
\begin{tikzpicture}[scale=0.85]
 \tikzstyle{quadri}=[circle,draw,text=black,thick]
 \tikzstyle{estun}=[->,>=latex,very thick]
 \node[quadri] (R1) at (0,3) {$R_1$};
 \node[quadri, fill=gray!25] (R2) at (-2,1) {$R_2$};
 \node[quadri] (R3) at (0,1) {$R_3$};
 \node[quadri] (R4) at (2,1) {$R_4$};
  \node[quadri,fill=gray!25] (P2) at (-0.5,-1) {$P_2$};
  \node[quadri,fill=gray!25] (R5) at (0.5,-1) {$R_5$};
  \node[quadri,fill=gray!25] (R6) at (1.5,-1) {$R_6$};
  \node[quadri,fill=gray!25] (R7) at (2.5,-1) {$R_7$};
  \node[quadri] (R8) at (4,3) {$R_8$};
  \node[quadri,fill=gray!25] (P6) at (3.5,1) {$P_6$};
  \node[quadri,fill=gray!25] (R9) at (4.5,1) {$R_9$};
  \node[quadri,fill=gray!25] (P8) at (6,3) {$P_8$};
  \draw[estun] (R1)--(R2);
  \draw[estun] (R1)--(R3);
  \draw[estun] (R1)--(R4);
  \draw[estun] (R3)--(R5);
  \draw[estun] (R3)--(P2);
    \draw[estun] (R4)--(R6);
  \draw[estun] (R4)--(R7);
 \draw[estun] (R8)--(R9);
  \draw[estun] (R8)--(P6);
\end{tikzpicture}
\end{center}
\caption[Graph of the completed reference family of Example~\ref{ex:base}]{
Graph corresponding to the completed version $\Rfam^\oplus$ of the reference family $\Rfam$ given by Example~\ref{ex:base} with the atoms given in Example~\ref{ex:base2}.\label{fig:completedgraph}}
\end{figure}
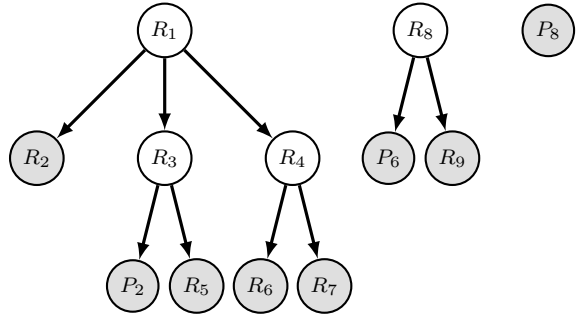

\subsection{Deriving the optimal post hoc bound}

The next result shows that the expression of the optimal post hoc bound $V^*_{\Rfam}$ can be simplified when
$\Rfam$ satisfies \eqref{eq:Forest}.

\begin{theorem}\label{thm:Forest}
Let $\Rfam$ be a reference family having the structure~\eqref{eq:Forest}. 
Then the optimal bound $V^*_{\Rfam}$ \eqref{eq:optbound1} can be derived from the bounds $\Vtildeq$ \eqref{eq:Vtildeq} and $\Vtilde$ \eqref{eq:VtildeK} in the following way:
\begin{align}
V^*_{\Rfam}(S)&= \Vtilde(S), \:\:\:S\subset \Nm;\label{itemi}\\
V^*_{\Rfam}(S)&= \Vtilde^d(S), \:\:\:S\subset \Nm,\label{itemii}
\end{align}
where $d$ is the maximum number of disjoint sets that can be found in the reference family, that is,
$$
d=\max \{ |Q|,  Q \subset \setK : \forall k,\tk\in Q , \:\:  k\neq\tk \Rightarrow R_k \cap R_\tk=\varnothing  \}.
$$
\end{theorem}


A byproduct of Theorem~\ref{thm:Forest} is that, if \eqref{equ:nested} holds, $V^*_{\Rfam}=\Vtilde^1(S)=\Vbar$ and we recover Proposition 2.5 of BNR. Another interesting case is the structure \eqref{equ:disjoint}, where $\Vtilde$ has a simpler form. This is summarized in the following result.

\begin{corollary}\label{cor:disjoint}
Let $\Rfam$ be a reference family.
\begin{itemize}
\item[(i)] if $\Rfam$ satisfies \eqref{equ:nested}, then $V^*_{\Rfam}=\Vbar$.
\item[(ii)] if $\Rfam$ satisfies \eqref{equ:disjoint}, then
$V^*_{\Rfam}(S) =  \sum_{k=1}^K \zeta_k\wedge |S\cap R_k| + \big|S\setminus\bigcup_{k=1}^K R_k  \big| ,\:\:\: S\subset \Nm.
$
\end{itemize}
\end{corollary}

Theorem~\ref{thm:Forest} and Corollary~\ref{cor:disjoint} are respectively proved in Section~\ref{sec:proof:thm} and Section~\ref{sec:proof:cor}.

The proof of Theorem~\ref{thm:Forest} being constructive, it provides an algorithm to compute easily $V^*_{\Rfam}(S)$, that we now describe.
%
%
Let us first introduce an additional piece of notation. For some reference family $\Rfam=(\Rf{i}{j},\zeta_{i,j})_{(i,j)\in\setK}$ of depth function $\phi$ (see \eqref{eq:depth}), we denote
\begin{equation*}
\setK^h=\{ (i,j)\in\setK : \h(i,j)=h \text{ or } (i=j \text{ and } \h(i,i)\leq h )      \}, \:\:\:h\geq 1.
\end{equation*}
Hence, each $\setK^h$ contains the indexes of the sets of depth $h$ and also the atoms with an inferior depth. Figure~\ref{fig:graphKtruc} displays some $\setK^h$ for the reference family of Example~\ref{ex:base}.
%

\begin{figure}[h!]
  \begin{center}
    \begin{tiny}
      \begin{tikzpicture}[scale=0.65]
        \tikzstyle{quadri}=[circle,draw,text=black, thick]
        \tikzstyle{estun}=[->,>=latex,very thick]
        \node[quadri,fill=orange] (R1) at (0,3) {$R_1$};
        \node[quadri] (R2) at (-2,1) {$R_2$};
        \node[quadri] (R3) at (0,1) {$R_3$};
        \node[quadri] (R4) at (2,1) {$R_4$};
        \node[quadri] (R5) at (0,-1) {$R_5$};
        \node[quadri] (R6) at (1.5,-1) {$R_6$};
        \node[quadri] (R7) at (2.5,-1) {$R_7$};
        \node[quadri,fill=orange] (R8) at (4,3) {$R_8$};
        \node[quadri] (R9) at (4,1) {$R_9$};
        \draw[estun] (R1)--(R2);
        \draw[estun] (R1)--(R3);
        \draw[estun] (R1)--(R4);
        \draw[estun] (R3)--(R5);
        \draw[estun] (R4)--(R6);
        \draw[estun] (R4)--(R7);
        \draw[estun] (R8)--(R9);
        \node[draw,label=above:{\large $\setK_1$},fit=(R1)(R2)(R3)(R4)(R5)(R6)(R7)(R8)(R9)] {};
      \end{tikzpicture}
      \begin{tikzpicture}[scale=0.65]
        \tikzstyle{quadri}=[circle,draw,text=black, thick]
        \tikzstyle{estun}=[->,>=latex,very thick]
        \node[quadri] (R1) at (0,3) {$R_1$};
        \node[quadri,fill=orange] (R2) at (-2,1) {$R_2$};
        \node[quadri,fill=orange] (R3) at (0,1) {$R_3$};
        \node[quadri,fill=orange] (R4) at (2,1) {$R_4$};
        \node[quadri] (R5) at (0,-1) {$R_5$};
        \node[quadri] (R6) at (1.5,-1) {$R_6$};
        \node[quadri] (R7) at (2.5,-1) {$R_7$};
        \node[quadri] (R8) at (4,3) {$R_8$};
        \node[quadri,fill=orange] (R9) at (4,1) {$R_9$};
        \draw[estun] (R1)--(R2);
        \draw[estun] (R1)--(R3);
        \draw[estun] (R1)--(R4);
        \draw[estun] (R3)--(R5);
        \draw[estun] (R4)--(R6);
        \draw[estun] (R4)--(R7);
        \draw[estun] (R8)--(R9);
        \node[draw,label=above:{\large $\setK_2$},fit=(R1)(R2)(R3)(R4)(R5)(R6)(R7)(R8)(R9)] {};
      \end{tikzpicture}
      \begin{tikzpicture}[scale=0.65]
        \tikzstyle{quadri}=[circle,draw,text=black, thick]
        \tikzstyle{estun}=[->,>=latex,very thick]
        \node[quadri] (R1) at (0,3) {$R_1$};
        \node[quadri,fill=orange] (R2) at (-2,1) {$R_2$};
        \node[quadri] (R3) at (0,1) {$R_3$};
        \node[quadri] (R4) at (2,1) {$R_4$};
        \node[quadri,fill=orange] (R5) at (0,-1) {$R_5$};
        \node[quadri,fill=orange] (R6) at (1.5,-1) {$R_6$};
        \node[quadri,fill=orange] (R7) at (2.5,-1) {$R_7$};
        \node[quadri] (R8) at (4,3) {$R_8$};
        \node[quadri,fill=orange] (R9) at (4,1) {$R_9$};
        \draw[estun] (R1)--(R2);
        \draw[estun] (R1)--(R3);
        \draw[estun] (R1)--(R4);
        \draw[estun] (R3)--(R5);
        \draw[estun] (R4)--(R6);
        \draw[estun] (R4)--(R7);
        \draw[estun] (R8)--(R9);
        \node[draw,label=above:{\large $\setK_3$},fit=(R1)(R2)(R3)(R4)(R5)(R6)(R7)(R8)(R9)] {};
      \end{tikzpicture}
    \end{tiny}
  \end{center}
  \caption[Nodes depicting $\setK^1$, $\setK^2$, $\setK^3$ for the reference family of Example~\ref{ex:base}]{
    Display of the nodes corresponding to $\setK^1$, $\setK^2$, $\setK^3$ (in orange) for the reference family given in Example~\ref{ex:base}.
    \label{fig:graphKtruc}}
\end{figure}
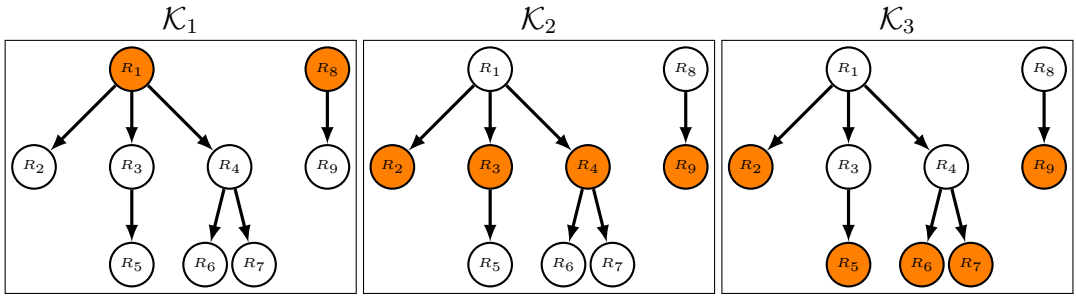

Algorithm~\ref{algo:Vstar} below gives the steps to compute $V^*_{\Rfam}(S)$: first, complete the family $\Rfam$ by adding all the members of the partition, as explained in Definition~\ref{def:complet}, in order to get $\Rfam^\oplus$. By Lemma~\ref{lm:completed}, we have $V^*_{\Rfam^\oplus}(S)=V^*_{\Rfam}(S)$, so that this operation does not change the targeted quantity. In particular, \eqref{eq:atom} holds after this step. Second, the algorithm uses a reverse loop, which successively updates a vector $V$ whose components correspond to active nodes; the current value of the bound is equal to the sum of the components of $V$. Each step of the loop will update the value of $V$ to make the bound possibly smaller, to obtain at the end $V^*_{\Rfam}(S)$. The time complexity of the Algorithm~\ref{algo:Vstar} for a given $S$ is $O(Hm)$, where $H=\max_{k\in\setK}\h(k)$ is the maximal depth of the reference family. where $\phi$ is the depth function defined by~\eqref{eq:depth}.

Let us describe the loop in more detail by using the particular situation of Figure~\ref{fig:graphKtruccomplet}.
Initialization: $H=3$ and $\setK^H=\setK^3$, which corresponds to the active nodes in the rightmost graph. Hence, $V$ is equal to the vector of values $\zeta_{k}\wedge|S\cap R_{k}|$ among these nodes. First step:  $h=2$ hence $\setK^h=\setK^2$, for which the active nodes are displayed in the middle graph. Each of these nodes $k\in \setK^2$, gives a bound  $\zeta_{k}\wedge|S\cap R_{k}|$ that should be compared with the one of the previous step, that is,  $ \sum_{k'\in Succ_k} V_{k'} $, where $Succ_k$ denotes the offspring of $R_k$. The vector $V$ is defined by the best choice among these two. Second (and final) step: $h=1$ hence $\setK^h=\setK^1$ (leftmost graph) which only contains the roots of the forest and where $V$ is updated following the same process. The algorithm then returns $V^*_{\Rfam}(S)=\sum_{k\in\setK^1} V_k$.

\begin{algorithm}
\KwData{$\Rfam=(\Rf{i}{j},\zeta_{i,j})_{(i,j)\in\setK}$ and $S\subset \Nm$.}
\KwResult{$V^*_{\Rfam}(S)$.}
$\Rfam \longleftarrow \Rfam^\oplus$; $\setK \longleftarrow \setK^\oplus$ (completion, see Definition~\ref{def:complet})\;
$ H \longleftarrow \max_{k\in\setK} \h(k)  $, see \eqref{eq:depth}\;
$V \longleftarrow (\zeta_{k}\wedge|S\cap R_{k}|)_{k\in\setK^H}$\;
\For{$h\in \{H-1,\dotsc,1\} $}{
$newV\longleftarrow (0)_{k \in  \setK^h}$\;
	\For{$k \in  \setK^h$}{
		$Succ_k \longleftarrow \{ k' \in  \setK^{h+1} : R_{k'}\subset R_k\}$\;
		$newV_k \longleftarrow \min\left( \zeta_{k}\wedge|S\cap R_{k}| ,  \sum_{k'\in Succ_k} V_{k'}   \right)$\;
	}
$V\longleftarrow newV$\;
}
\Return $\sum_{k\in\setK^1} V_k$.
\caption[Computation of \texorpdfstring{$V^*_{\Rfam}(S)$}{Vstar}]{Computation of $V^*_{\Rfam}(S)$}
\label{algo:Vstar}
\end{algorithm}

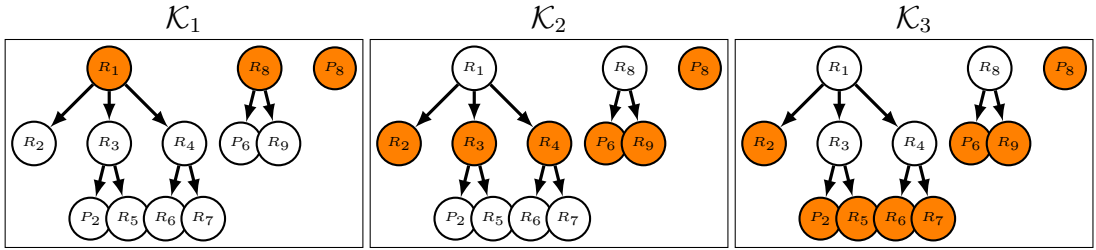
\begin{figure}[h!]
  \begin{center}
    \begin{tiny}
      \begin{tikzpicture}[scale=0.5]
        \tikzstyle{quadri}=[circle,fill=white,draw,text=black,thick]
        \tikzstyle{estun}=[->,>=latex,very thick]
        \node[quadri, fill=orange] (R1) at (0,3) {$R_1$};
        \node[quadri] (R2) at (-2,1) {$R_2$};
        \node[quadri] (R3) at (0,1) {$R_3$};
        \node[quadri] (R4) at (2,1) {$R_4$};
        \node[quadri] (P2) at (-0.5,-1) {$P_2$};
        \node[quadri] (R5) at (0.5,-1) {$R_5$};
        \node[quadri] (R6) at (1.5,-1) {$R_6$};
        \node[quadri] (R7) at (2.5,-1) {$R_7$};
        \node[quadri, fill=orange] (R8) at (4,3) {$R_8$};
        \node[quadri] (P6) at (3.5,1) {$P_6$};
        \node[quadri] (R9) at (4.5,1) {$R_9$};
        \node[quadri,fill=orange] (P8) at (6,3) {$P_8$};
        \draw[estun] (R1)--(R2);
        \draw[estun] (R1)--(R3);
        \draw[estun] (R1)--(R4);
        \draw[estun] (R3)--(R5);
        \draw[estun] (R3)--(P2);
        \draw[estun] (R4)--(R6);
        \draw[estun] (R4)--(R7);
        \draw[estun] (R8)--(R9);
        \draw[estun] (R8)--(P6);
        \node[draw,label=above:{\large $\setK_1$},fit=(R1)(R2)(R3)(R4)(R5)(R6)(R7)(R8)(R9)(P8)] {};
      \end{tikzpicture}
      \begin{tikzpicture}[scale=0.5]
        \tikzstyle{quadri}=[circle,fill=white,draw,text=black,thick]
        \tikzstyle{estun}=[->,>=latex,very thick]
        \node[quadri] (R1) at (0,3) {$R_1$};
        \node[quadri, ,fill=orange] (R2) at (-2,1) {$R_2$};
        \node[quadri,fill=orange] (R3) at (0,1) {$R_3$};
        \node[quadri,fill=orange] (R4) at (2,1) {$R_4$};
        \node[quadri] (P2) at (-0.5,-1) {$P_2$};
        \node[quadri] (R5) at (0.5,-1) {$R_5$};
        \node[quadri] (R6) at (1.5,-1) {$R_6$};
        \node[quadri] (R7) at (2.5,-1) {$R_7$};
        \node[quadri] (R8) at (4,3) {$R_8$};
        \node[quadri,fill=orange] (P6) at (3.5,1) {$P_6$};
        \node[quadri,fill=orange] (R9) at (4.5,1) {$R_9$};
        \node[quadri,fill=orange] (P8) at (6,3) {$P_8$};
        \draw[estun] (R1)--(R2);
        \draw[estun] (R1)--(R3);
        \draw[estun] (R1)--(R4);
        \draw[estun] (R3)--(R5);
        \draw[estun] (R3)--(P2);
        \draw[estun] (R4)--(R6);
        \draw[estun] (R4)--(R7);
        \draw[estun] (R8)--(R9);
        \draw[estun] (R8)--(P6);
        \node[draw,label=above:{\large $\setK_2$},fit=(R1)(R2)(R3)(R4)(R5)(R6)(R7)(R8)(R9)(P8)] {};
      \end{tikzpicture}
      \begin{tikzpicture}[scale=0.5]
        \tikzstyle{quadri}=[circle,fill=white,draw,text=black,thick]
        \tikzstyle{estun}=[->,>=latex,very thick]
        \node[quadri] (R1) at (0,3) {$R_1$};
        \node[quadri, ,fill=orange] (R2) at (-2,1) {$R_2$};
        \node[quadri] (R3) at (0,1) {$R_3$};
        \node[quadri] (R4) at (2,1) {$R_4$};
        \node[quadri,,fill=orange] (P2) at (-0.5,-1) {$P_2$};
        \node[quadri,,fill=orange] (R5) at (0.5,-1) {$R_5$};
        \node[quadri,,fill=orange] (R6) at (1.5,-1) {$R_6$};
        \node[quadri,fill=orange] (R7) at (2.5,-1) {$R_7$};
        \node[quadri] (R8) at (4,3) {$R_8$};
        \node[quadri,fill=orange] (P6) at (3.5,1) {$P_6$};
        \node[quadri,fill=orange] (R9) at (4.5,1) {$R_9$};
        \node[quadri,fill=orange] (P8) at (6,3) {$P_8$};
        \draw[estun] (R1)--(R2);
        \draw[estun] (R1)--(R3);
        \draw[estun] (R1)--(R4);
        \draw[estun] (R3)--(R5);
        \draw[estun] (R3)--(P2);
        \draw[estun] (R4)--(R6);
        \draw[estun] (R4)--(R7);
        \draw[estun] (R8)--(R9);
        \draw[estun] (R8)--(P6);
        \node[draw,label=above:{\large $\setK_3$},fit=(R1)(R2)(R3)(R4)(R5)(R6)(R7)(R8)(R9)(P8)] {};
      \end{tikzpicture}
    \end{tiny}
  \end{center}
  \caption[Same as Figure~\ref{fig:graphKtruc} but for the completed version]{Same as Figure~\ref{fig:graphKtruc} but for the completed version.
    \label{fig:graphKtruccomplet}}
\end{figure}

\section{Local calibration of the reference family}\label{sec:calibration}

In this section, we explain how to build a reference family $\Rfam$ such that \eqref{eq:jercontrol} holds.
The results presented in this section hold for any deterministic $(R_k)_k$ and the calibration concerns only $(\zeta_k)_k$ here.


\subsection{Calibration of $\zeta_k$ by DKW inequality}
\label{sec:calibr-zeta_k-dkw}

In this section, we estimate $|S \cap \cH_0|$ by using an approach close in spirit to the so-called Storey estimator \citep{storey2002direct}. The latter depends on a parameter, denoted by $t$ here, that has to be chosen appropriately (see \citealp{blanchard2009adaptive} for a discussion on this issue).
To avoid this caveat while improving accuracy, we can derive an estimator uniform on $t$ by using the DKW inequality \citep{dvoretzky1956asymptotic}, with the optimal constant of \citet{massart1990tight}.

For any deterministic subsets $R_k\subset \Nm$, $k\in\setK$, $K=|\setK|$, let
\begin{equation}\label{equ:zeta}
\zeta_k(X)= |R_k|\wedge \min_{t\in[0,1)} \left\lfloor \frac{C}{2(1-t)} +
\left(\frac{C^2}{4(1-t)^2} + \frac{\sum_{i\in R_k} \mathbf{1}{\{p_i(X) > t\}}}{1-t}\right)^{1/2} \right\rfloor^2      , \:\:\:\:\:k\in \setK ,
\end{equation}
where  $C=\sqrt{\frac{1}{2}\log \left(\frac{K}{\alpha} \right) }$ and $\lfloor x\rfloor$ denotes the largest integer smaller than or equal to $x$.

\begin{proposition}\label{prop:calibration}
Consider any deterministic (different) subsets $R_k\subset \Nm$, $k\in\setK$ ($K=|\setK|$) and assume $\alpha/K< 1/2$.  Assume that for all $k\in \setK$, the $p$-value family $\{p_i(X), \:i \in R_k\}$ satisfies \eqref{eq:superunif} and \eqref{eq:indep}.
Then the JER control \eqref{eq:jercontrol} holds for the reference family
$\Rfam=(R_k,\zeta_k(X))_{k\in\setK}$, for which the local bounds $\zeta_k$ are given by \eqref{equ:zeta}.
\end{proposition}

Combining Proposition~\ref{prop:calibration}  with Lemma~\ref{lem:Vtilde}, we obtain that, under the assumptions of Proposition~\ref{prop:calibration}, the bound
\begin{align}\label{Vnew}
\VDKW=\Vtilde \mbox{ given by \eqref{eq:VtildeK} with } \Rfam=(R_k,\zeta_k(X))_{k\in\setK} \mbox{ and  $\zeta_k(X)$ given by \eqref{equ:zeta},}
\end{align}
 satisfies \eqref{aim} and thus is a valid post hoc bound.

Proposition~\ref{prop:calibration} is proved in Section~\ref{sub:proof:calibration}. Note that $\zeta_k(X)\geq \lfloor\log(K/\alpha)/2\rfloor\geq1$ as soon as $\alpha\leq e^{-2}K$. Hence, this contrasts with previous approaches~\citep{blanchard2017post,goeman2011multiple}, for which $\zeta_k=0$ was included in the reference family.
This means that using this reference family induces a minimum cost.
In the next section, we will see that this cost is generally compensated by the accuracy of the joint estimation of $|R_k\cap \cH_0|$, $k\in \mathcal{K}$.

\begin{remark}
In practice, $\zeta_k(X)$ in \eqref{equ:zeta} can be computed as
\begin{align*}
\zeta_k(X)=s\wedge\min_{ 0\leq \l \leq s} \left\lfloor \frac{C}{2(1-p_{(\l)})} +
\left(\frac{C^2}{4(1-p_{(\l)})^2} + \frac{s-\l}{1-p_{(\l)}} \right)^{1/2} \right\rfloor^2,
\end{align*}
where $s=|R_k|$ and
$
0=p_{(0)}\leq p_{(1)} \leq \dots \leq p_{(s)}
$
are the ordered $p$-values of $\{p_{i}(X),i\in R_k\}$.
\end{remark}

\begin{remark}\label{rem:comparisonpreviouspi0}
With our notation, the previous  $(1-\alpha)$-confidence bound of  \citet{genovese2004stochastic} (Equation~(16) therein) corresponds to take
$$
\zeta_k^{GW}(X)= |R_k|\wedge \min_{t\in[0,1)} \left\lfloor\frac{\sum_{i\in R_k} \mathbf{1}{\{p_i(X) > t\}}+|R_k|^{1/2} C}{1-t} \right\rfloor.
$$
By using \eqref{equ:majdwedge} in Lemma~\ref{lem:trivial} with $a=1-t$, $b=C$, $c=\sum_{i\in R_k} \mathbf{1}{\{p_i(X) > t\}}$,  and $d=|R_k|$,
we can see that the quantity $\zeta_k^{GW}(X)$ is always  larger than the  $\zeta_k(X)$ given by \eqref{equ:zeta}.
Hence our result is a uniform improvement of  \citet{genovese2004stochastic}.
%
%
%
%
%
%
\end{remark}

\begin{remark}
  \label{rem:influence-alpha-DKW}
The local bounds $\zeta_k$ in \eqref{equ:zeta} depend on the target level $\alpha$ only through $C$, where $2C^2=\log (K/\alpha)$. Therefore, the post hoc bounds derived from Proposition~\ref{prop:calibration} are expected to depend only weakly on $\alpha$. This important point is illustrated in our numerical experiments (Section \ref{sec:num}), where this property is used to propose a hybrid post hoc bound taking the best of both the Simes and the DKW-based bounds.
\end{remark}

\subsection{Comparison to existing post hoc bounds}
\label{sec:comp-class-bounds}
To explore the benefit of the new reference family when the signal is localized, let us consider a stylized model where the signal is localized according to a regular partition
\begin{equation}\label{equ:ref:parti}
R_k=\{1+(k-1)s,\dots, ks\}, \:\:1\leq k \leq K,
\end{equation}
 composed of $K$ regions of equal size $s$. In particular, this reference family satisfies \eqref{equ:disjoint}.
Among the regions $R_k$,  only $R_1$ contains false nulls, and  $r\in(0,1)$ denotes the proportion of signal in $R_1$, that is
\begin{align}
  \label{eq:def-r}
  r & = |R_1\cap\mathcal{H}_1|/|R_1|.
\end{align}
The remaining regions contain no signal, that is $|R_k\cap\mathcal{H}_1|=0$, for $k\geq 2$.\\

In addition, we consider an independent Gaussian one-sided setting where the false nulls have mean $\mu>0$, that is, we assume that $X_i\sim\mathcal{N}(0, 1)$ if $i\in\mathcal{H}_0$ and $X_i\sim\mathcal{N}(\mu, 1)$ if $i\in\mathcal{H}_1$, and the $p$-values are derived as $p_i(X)=\bar\Phi(X_i)$, $i \in \Nm$, where $\bar\Phi$ denotes the upper-tail of the standard normal distribution.


\begin{proposition}\label{prop:comparison}
Let us consider the post hoc bounds $\VBonf$ \eqref{equ:BonfV}; $\VSimes$ \eqref{equ:SimesV} and the new post hoc bound $\VDKW$ given by \eqref{Vnew} and associated to the reference regions $R_k$ defined above.
In the setting defined above, we have
\begin{align}
\frac{\E(\VDKW(R_1))}{|R_1|}
&\leq 1 \wedge \left(
1-r + 2 r \:\ol{\Phi}(\mu) + \frac{4C}{\sqrt{s}}\left(1+\frac{C}{\sqrt{s}}  \right) \right)
\label{boundVDKW}
\\
\frac{\E( \VSimes(R_1))}{|R_1|}&\geq (1-r)(1-\alpha s/m) + r\: \ol{\Phi}(\mu -\ol{\Phi}^{-1}(\alpha s/m));\label{boundSimes}\\
\frac{\E(\VBonf(R_1))}{|R_1|}
&=  (1-r)(1-\alpha/m) + r\: \ol{\Phi}(\mu-\ol{\Phi}^{-1}(\alpha/m))\label{boundVBonf}.
\end{align}
\end{proposition}
This proposition is proved in Section~\ref{sub:proof:comparison}. In particular, combining \eqref{boundVDKW} and \eqref{boundSimes} yields
\begin{equation}\label{equ:ratiobound}
\frac{\E(\VDKW(R_1))}{\E( \VSimes(R_1))}
\leq
\frac{1 \wedge \left(
1-r + 2 r \:\ol{\Phi}(\mu) + \frac{4C}{\sqrt{s}}\left(1+\frac{C}{\sqrt{s}}  \right)\right)}{(1-r)(1-\alpha s/m) + r\: \ol{\Phi}(\mu -\ol{\Phi}^{-1}(\alpha s/m))}.
\end{equation}
This ratio is displayed in Figure~\ref{IllustrationBound} for a choice of model parameters. The new bound can substantially improve the Simes bound over a wide range of effect sizes.

\begin{figure}[htp!]
  \centering
 \includegraphics[scale=0.5]{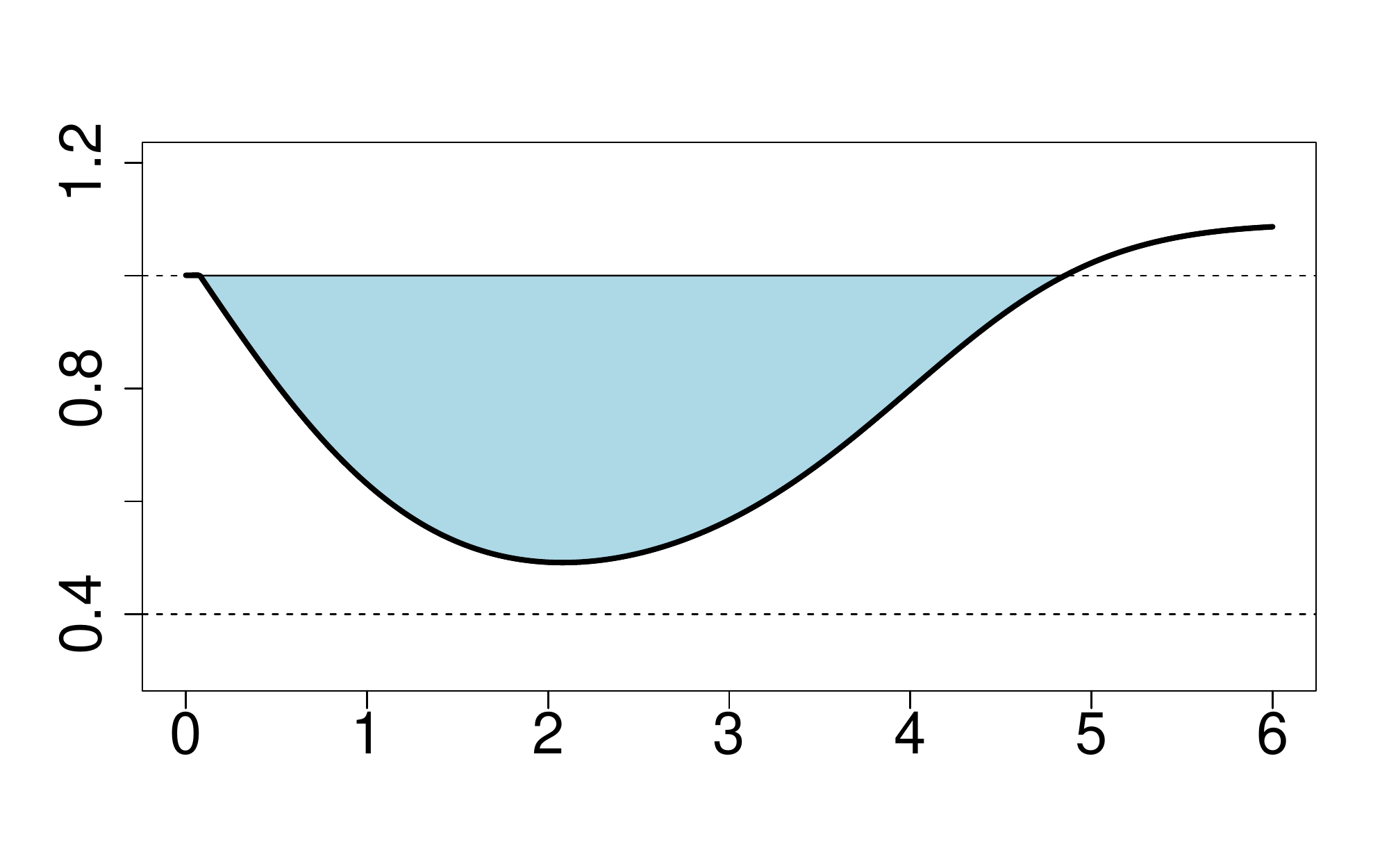}
 \vspace{-1cm}
\caption[Upper bound of the ratio between the new bound and the Simes bound against effect size]{$Y$-axis: upper bound of the ratio between the new bound and the Simes bound, see \eqref{equ:ratiobound}.
$X$-axis: effect size $\mu$. $m=10^7$, $s=m^{2/3}$, $K=m/s$, $r=3/5$, $\alpha=0.1$.}  \label{IllustrationBound}
\end{figure}

This improvement can also be put forward by an asymptotic approach.

\begin{corollary}
Let us consider the framework of Proposition~\ref{prop:comparison}. In the asymptotic setting in $m$ where $s$ tends to infinity with $s\gg\log K$ and $\mu$ tends to infinity with $\mu -  \ol{\Phi}^{-1}(\alpha/m) \rightarrow -\infty$, we have
\begin{align*}
\limsup_m \left\{ \frac{\E(\VDKW(R_1))}{|R_1|} \right\} \leq 1-r, \:\:\mbox{ and }\:\: \limsup_m \left\{ \frac{\E(\VBonf(R_1))}{|R_1|} \right\}= 1.
\end{align*}
If moreover $s \ll m$ (i.e., $K\rightarrow\infty$) and $\mu -  \ol{\Phi}^{-1}(\alpha s/m) \rightarrow -\infty$, we have
\begin{align*}
\limsup_m \left\{ \frac{\E(\VDKW(R_1))}{|R_1|} \right\} \leq 1-r, \:\:\mbox{ and }\:\: \limsup_m \left\{ \frac{\E(\VSimes(R_1))}{|R_1|} \right\}= 1.
\end{align*}
\end{corollary}

In particular, this corollary establishes that the order of the new bound can improve the Simes bound by a factor $1-r$.

\section{Numerical experiments}\label{sec:num}

\subsection{Setting}
In this section we perform  numerical experiments to compare
our new post hoc bound $\VDKW$ \eqref{Vnew}  with Simes post hoc bound \eqref{equ:SimesV}.
Let $q$ be some fixed integer, say larger than $1$. We consider two versions of our new bound:
\begin{itemize}
\item
The first version of our post hoc bound, denoted $V_{\rm part}$, is defined by \eqref{Vnew} in which the reference family $\Rfam^{\rm part}$ is the regular partition of $\Nm$ given by \eqref{equ:ref:parti} for $K^{\rm part}=2^q$ ($s=m/2^q$ being assumed to be an integer).
\item
The second version of our post hoc bound, denoted $V_{\rm tree}$, is defined similarly by \eqref{Vnew}, but  the reference family $\Rfam^{\rm tree}$ is given this time by the perfect binary tree whose leaves are the elements of $\Rfam^{\rm part}$. Hence, by using the notation of Lemma~\ref{lm:partition}, this means $P_k=\{1+(k-1)s,\dots, ks\},$ $1\leq k \leq 2^q$. The cardinal of the reference family is thus $K^{\rm tree} = 2^{q+1}-1$.
\end{itemize}
The true/false null hypothesis configuration is as follows:
the false null hypotheses are contained in $P_k$ for $1 \leq k \leq K_1$, for some fixed value of $K_1$. The quantity $r$ is defined similarly as in \eqref{eq:def-r}, as the fraction of false null hypotheses in those $P_k$, 
 and is set to $r \in \{0.5, 0.75, 0.9, 1\}$. All of the other partition pieces only contain true null hypotheses. Finally,  the true null $p$-values are distributed as i.i.d. $\cN(0,1)$, and false null $p$-values are distributed as i.i.d. $\cN(\bar{\mu}, 1)$, where $\bar{\mu}$ is a fixed value in $\{2, 3, 4\}$.  This construction is illustrated in Figure~\ref{fig:graph-simu} for $q=3$ (leading to $K^{\rm part}=8$ and $K^{\rm tree}=15$) and $K_1=2$. In our experiments, we have chosen $q=7$ and $s=100$ (corresponding to $K^{\rm part}=128$ and $K^{\rm tree}=255$ and $m=12800$), and $K_1=8$.

\begin{figure}[h!]
\begin{center}
\begin{tikzpicture}[scale=0.85]
 \tikzstyle{quadri}=[circle,draw,text=black, thick]
 \tikzstyle{quadriH1}=[circle,draw,text=black,fill=pink, thick]
 \tikzstyle{estun}=[<-,>=latex,very thick]

 \node[quadri] (R18) at (-0.5,6) {$P_{1:8}$};

 \node[quadri] (R14) at (-3.5,4) {$P_{1:4}$};
 \node[quadri] (R58) at (2.5,4) {$P_{5:8}$};

 \node[quadriH1] (R12) at (-5,2) {$P_{1:2}$};
 \node[quadri] (R34) at (-2,2) {$P_{3:4}$};
 \node[quadri] (R56) at (1,2) {$P_{5:6}$};
 \node[quadri] (R78) at (4,2) {$P_{7:8}$};

 \node[quadriH1] (R1) at (-6,0) {$P_{1}$};
 \node[quadriH1] (R2) at (-4,0) {$P_{2}$};
 \node[quadri] (R3) at (-3,0) {$P_{3}$};
 \node[quadri] (R4) at (-1,0) {$P_{4}$};
 \node[quadri] (R5) at (0,0) {$P_{5}$};
 \node[quadri] (R6) at (2,0) {$P_{6}$};
 \node[quadri] (R7) at (3,0) {$P_{7}$};
 \node[quadri] (R8) at (5,0) {$P_{8}$};

\node[draw,very thick,label=left:{Partition},dotted,fit=(R1)(R2)(R3)(R4)(R5)(R6)(R7)(R8)] (part) {};
\node[draw,label=left:{Tree},fit=(R1)(R2)(R3)(R4)(R5)(R6)(R7)(R8)(R12)(R34)(R56)(R78)(R14)(R58)(R18)] (tree) {};


  \draw[estun] (R1)--(R12);
  \draw[estun] (R2)--(R12);
  \draw[estun] (R3)--(R34);
  \draw[estun] (R4)--(R34);
  \draw[estun] (R5)--(R56);
  \draw[estun] (R6)--(R56);
  \draw[estun] (R7)--(R78);
  \draw[estun] (R8)--(R78);

  \draw[estun] (R12)--(R14);
  \draw[estun] (R34)--(R14);
  \draw[estun] (R56)--(R58);
  \draw[estun] (R78)--(R58);

  \draw[estun] (R14)--(R18);
  \draw[estun] (R58)--(R18);

\end{tikzpicture}
\end{center}
\caption[Partition and perfect binary tree structures used in simulations]{Partition and perfect binary tree structures used in simulations, here with $q=3$ and $K_1=2$ ($K^{\rm part}=8$ and $K^{\rm tree}=15$). The pink nodes are those containing some signal.}
\label{fig:graph-simu}
\end{figure}
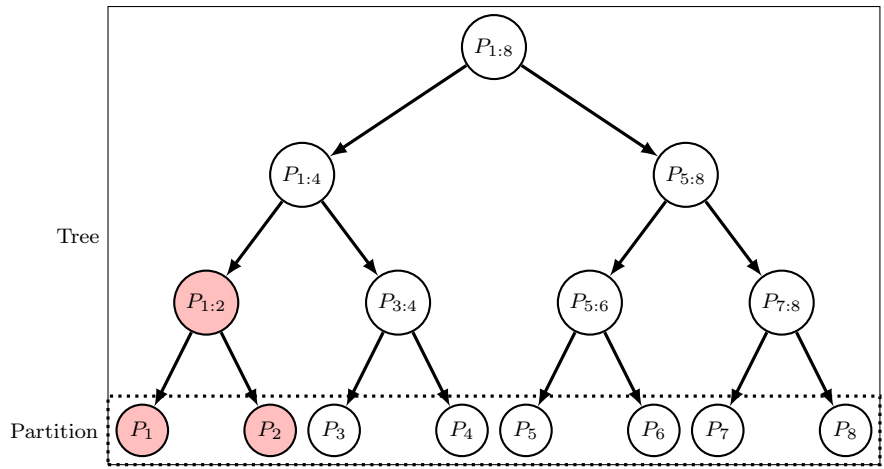

We also performed numerical experiments with $s \in \{10, 20, 50\}$ and $K_1 \in \{1, 4, 16\}$, and with Poisson- and Gaussian-distributed $\bar{\mu}$. Because the results are qualitatively similar, we only report the above-described setting.

\subsection{Comparing confidence envelopes}


One possible way to evaluate the performance of post hoc bounds is to consider the associated confidence envelopes on the number of true discoveries among the most significant hypotheses. Formally, for $k = 1, \dots, m$, we let $S_k = \{i_1, \dots, i_k\}$, where $i_j$ is the index of the $j^{\rm th}$ smallest $p$-value. Note that focusing on such sets is \textit{a priori} favorable to the Simes bound, for which the reference family are among the $S_k$. In Figure~\ref{fig:upper-bounds-alpha=0.05}, each panel corresponds to a particular choice of the model parameters $d$ (in rows) and $\bar{\mu}$ (in columns). Each panel compares the actual number of true positives
 $(k-|\cH_0\cap S_k|)$,
 $k = 1, \dots, m$ (labelled ``Oracle'') to  post hoc bounds of the form
$(k-V(S_k))$,
$k = 1, \dots, m$, where $V$ is $V_{\rm Simes}$, $V_{\rm part}$, or $V_{\rm tree}$.
 In this figure, the confidence level is set to $1-\alpha = 95\%$.
\begin{figure}[!h]
  \centering
\includegraphics[width=0.99\textwidth]{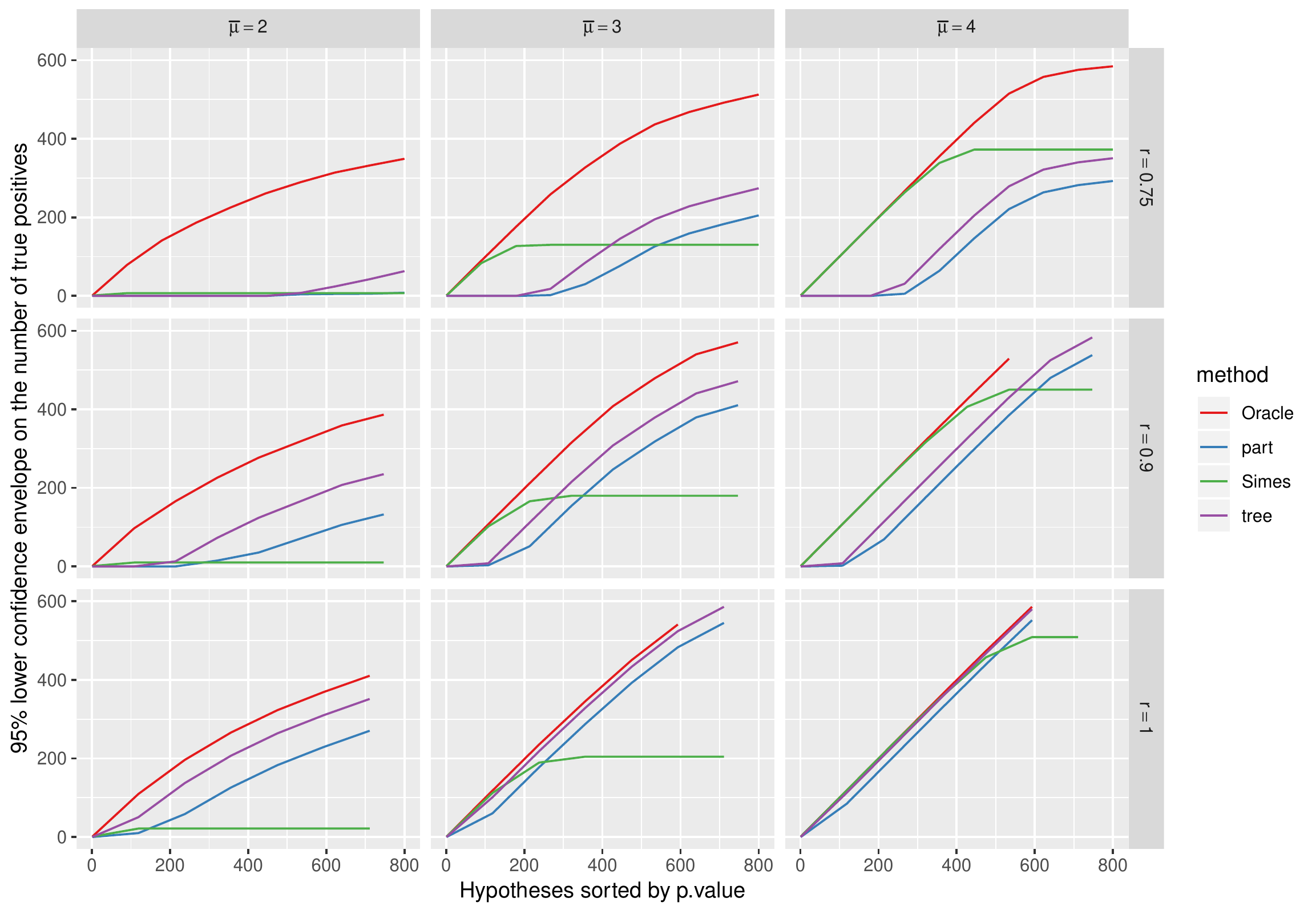}
\caption[95\texorpdfstring{\%}{pourcent} lower confidence envelopes from Simes inequality and from the proposed methods]{95\% lower confidence envelopes on the number of true positives obtained from Simes inequality and from the proposed methods are compared to the actual (Oracle) number of true positives.}
\label{fig:upper-bounds-alpha=0.05}
\end{figure}

The chosen model parameters span a wide range of situations between very low and very high signal. For very low signal ($\bar{\mu}=2, r=0.75$, top-left panel), all the bounds are trivial, i.e. provide $V(S_k)$ close to $\abs{S_k} (=k)$. As expected, all the bounds get sharper as the signal to noise ratio increases, that is, as $\bar{\mu}$ or $r$ increase, and for very high signal ($\bar{\mu}=4, r=1$, bottom-right panel), all the bounds are very close to the actual number of true positives. The tree-based bound dominates the partition-based bound, which is expected because in this particular experiment, the regions $P_k$ containing signal are adjacent (see Figure~\ref{fig:graph-simu}), and the
multiscale nature of the tree-based bound allows it to take advantage of large-scale clusters. When the signal regions are not adjacent, these two bounds are very close (additional numerical experiments not shown). Our proposed bounds are more sensitive to the proportion of signal in each active region, while the Simes bound is more sensitive to the strength of the signal in those regions. As a result, none of the Simes and the ``tree'' bound is uniformly better than the other one. The Simes bound is typically sharper than the ``tree'' bound for small values of $k$, but becomes more conservative for larger values of $k$.  This is expected, because the ``tree'' bound is based on  \emph{estimating the proportion of} non-null items, while the Simes bound is based on  \emph{pinpointing} non-null items.

\subsection{Hybrid approach}

An interesting question raised in Section~\ref{sec:calibr-zeta_k-dkw} (Remark~\ref{rem:influence-alpha-DKW}) is how these bounds are influenced by the target confidence level, which is fixed to  $1-\alpha = 95\%$ in Figure~\ref{fig:upper-bounds-alpha=0.05}. In Figure~\ref{fig:upper-bounds-alphas} we compare the bounds obtained across values of $\alpha$ (corresponding to different line types) for $\bar{\mu} \in \{3, 4\}$ and $r \in \{0.75, 0.9\}$. The influence of $\alpha$ on the Simes bound is quite substantial. This is consistent with the shape of the bound~\eqref{equ:SimesV}, the $p$-values are directly compared to $\alpha$. The influence of $\alpha$ on the bounds derived from \eqref{equ:zeta} is much weaker, as expected from  Remark \ref{rem:influence-alpha-DKW}. In particular, the envelopes derived from the ``tree'' method are very close  to each other when $\alpha$ varies from $0.001$ to $0.05$.
\begin{figure}[!h]
  \centering
\includegraphics[width=0.99\textwidth]{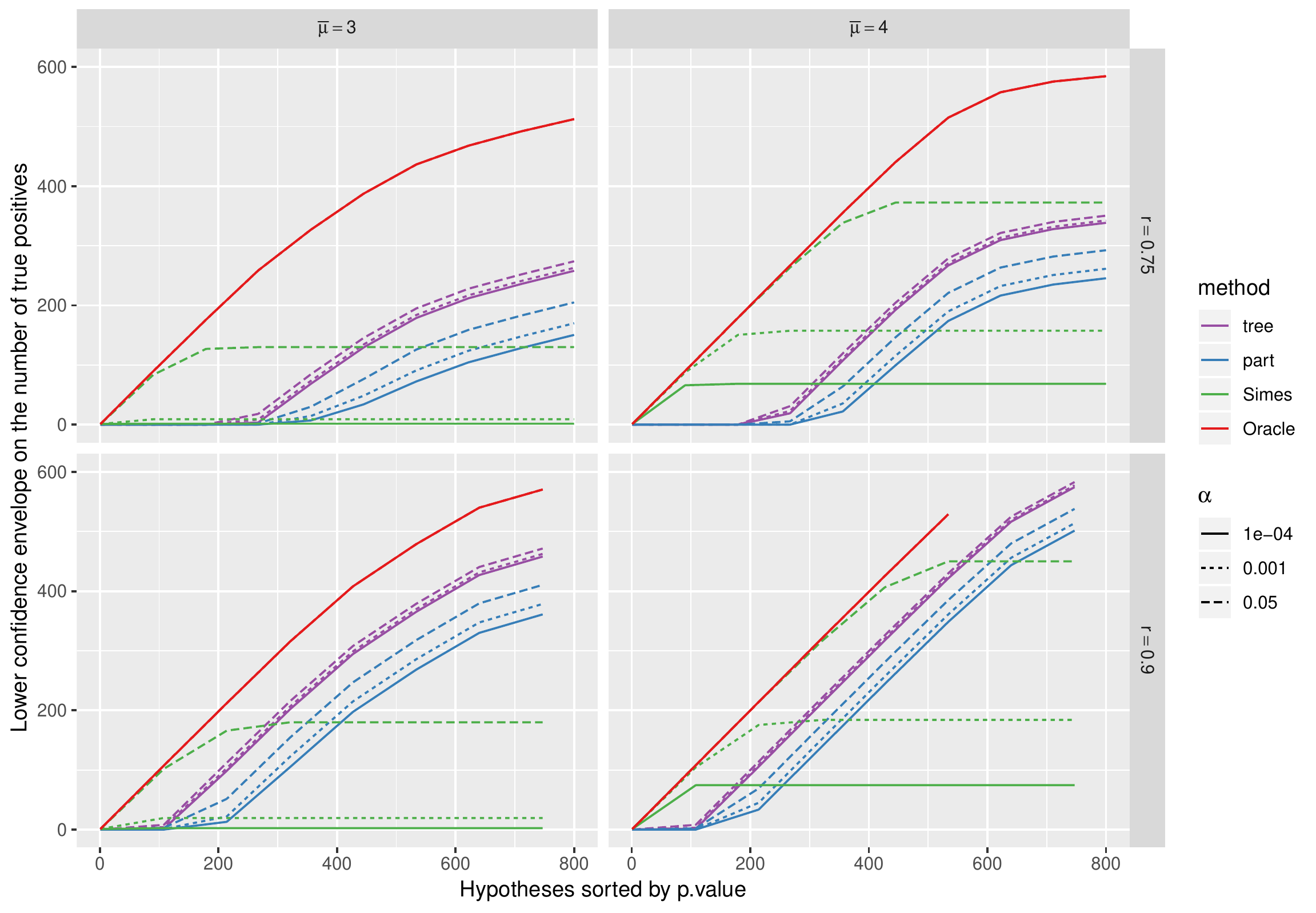}
\caption[Influence of the target level parameter \texorpdfstring{$\alpha$}{alpha} on upper confidence envelopes]{Influence of the target level parameter $\alpha$ on upper confidence envelopes on the number of true positives.}
\label{fig:upper-bounds-alphas}
\end{figure}
These striking differences suggest to introduce hybrid confidence envelopes that could take advantage of the superiority of the Simes bound on sets $S_k$ for small $k$ with that of the DKW-tree-based bound on sets $S_k$ for larger $k$. For a fixed $\gamma \in [0,1]$, let us define the bound $\Vhyb^\gamma$ as follows.  For $S \subset \Nm$,
\begin{align*}
  \Vhyb^\gamma(\alpha, S) = \min \paren{\VSimes((1-\gamma) \alpha, S) , V_{\rm tree}(\gamma\alpha, S)},
\end{align*}
where the notation in the bounds explicitly acknowledges the dependence of the bounds in the target level $\alpha$. By an union bound, $\Vhyb^\gamma(\alpha, \cdot)$ is a $(1-\alpha)$-level post hoc bound.
Figure \ref{fig:upper-bounds-hybrid} gives an illustration with $\alpha = 0.05$ and $\gamma = 0.02$.  In this case, the hybrid envelope is the minimum of the Simes envelope at level $(1-\gamma)\alpha = 0.049$ and the DKW-tree-based envelope at level $0.001$. Because $(1-\gamma)\alpha$ is very close to $\alpha$, the confidence envelope $ \Vhyb^{0.02}$ is essentially equivalent to the Simes-based confidence envelope for small $k$; for larger values of $k$,  $ \Vhyb^{0.02}$ is only slightly worse than the DKW-tree-based confidence envelope at level $\gamma \alpha = 0.001$.


\begin{figure}[!h]
  \centering
\includegraphics[width=0.99\textwidth]{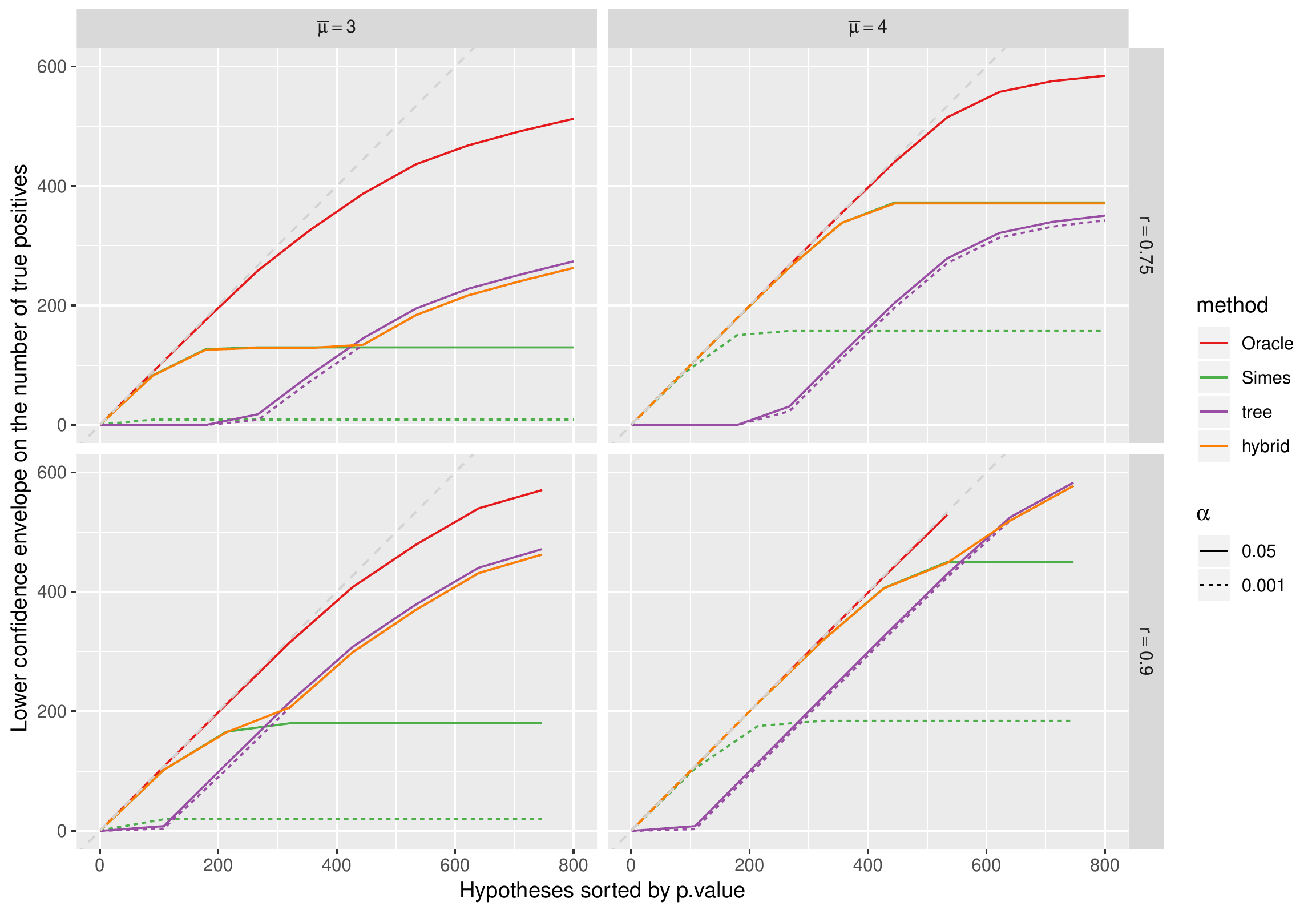}
\caption[Combining Simes and tree-based upper confidence bounds]{Combining Simes and tree-based confidence envelopes on the number of true positives into a hybrid confidence envelope.}
\label{fig:upper-bounds-hybrid}
\end{figure}

\section{Discussion}\label{sec:dis}

\subsection{Comparison to \citet{meijer2015region}}\label{sec:compaGoeman}

Since our aim is similar to the one of \citet{meijer2015region} (denoted MKG below for short), let us make a short qualitative comparison between MKG and our study. 
 First, while both approaches are based on graph-structured subsets $\{R_k,k\in\setK \}$, the geometrical shapes of the nodes $R_k$ are different:  the nodes in MKG correspond to all possible consecutive intervals, possibly overlapping, while our regions are based on partitioned regions at different resolutions. Our approach avoids redundancies of the tests but is suitable when the signal is structured according to the pre-specified partition structure, and may lead to a less accurate bound otherwise. This in turn impacts the way the local pieces of information are combined. The MKG approach uses a sequential, top-down algorithm, with an $\alpha$-recycling method (that allows, for instance, to spend the same nominal level $\alpha$ both for a parent and its child). By contrast, our approach uses a bottom-up algorithm, with an overall nominal level adjusted by a simple overall union bound, which is generally conservative but seems fair here as the nodes are disjoint (at each resolution). 
 
Second, the criteria used are different: MKG  focus on simultaneous FWER control of local tests of intersections of null hypotheses $\cap_{i\in R_k} H_{0,i}$, $k\in\setK$, while our statistical criterion ensures with high probability $|\cH_0\cap R_k| \leq \zeta_k$, for all $k\in\setK$, for some bounds $\zeta_k$. As already noted in BNR (see the supplementary file therein),  the two approaches coincide when  $\zeta_k=|R_k|-1$, because $|\cH_0\cap R_k| > |R_k|-1$ is equivalent to the fact that $\cap_{i\in R_k} H_{0,i}$ is true. Hence, a family  $\{R_k,k\in\setK \}$ violating $|\cH_0\cap R_k| > |R_k|-1$ for some $k$ will also wrongly reject $\cap_{i\in R_k} H_{0,i}$ for some $k$. However, when using another type of $\zeta_k$, such as the DKWM device used here, such a connection is not valid anymore and the two criteria does not incorporate the local structure of the nodes in the same way. Here, using $\zeta_k$'s based on classical estimators will in principle lead to better post hoc bounds.  

Third, within each node, the local statistics used are not of the same nature: in MKG, the local tests are based on a  multivariate $\chi^2$-type test, see \cite{goeman2004global}. Here, we use an estimator relying on individual $p$-values that exploits the independence structure. This means that the assumptions made in MKG are much weaker, since it is valid under arbitrary dependence. Our approach can in principle also accommodate such a distributional setting, but this needs additional investigations, see the discussion in Section~\ref{extensiondep}. 

Finally, let us mention a setting for which the two methods can be fairly compared. 
First take the MKG method with Bonferroni local tests. As proved in MKG, the resulting FWER controlling procedure (reject the $H_{0,i}$ for which $V(\{i\})=0$) then reduces to the Holm procedure \cite{holm1979simple}. 
By contrast, if we consider $\zeta_k$ equals to the number of accepted null hypotheses by the Holm procedure restricted to $R_k$ (satisfying \eqref{equ:disjoint}), 
our methodology induces another overall FWER controlling procedure: simply the one rejecting all the null hypotheses rejected by the local Holm procedures. Both FWER controlling procedures are valid under arbitrary independence. Interestingly, if the signal is sparse but localized in one of the pre-specified $R_k$, the new procedure will dominate the Holm procedure (this is supported by a numerical experiment and a theoretical study, not reported here for short). 
This illustrates, once again, that our methodology can improve the state of the art, even in a very elementary framework.  

\subsection{Extension to general local confidence bounds}\label{extensiondep}

In this work, the local bounds $\zeta_k$ have been designed by using the DKW inequality.  This can be straightforwardly extended  to the case where the bound \eqref{equ:zeta} is replaced by $\zeta_k(X)=L_k(\alpha/K)$, for which the function  $L_k(\cdot)$ is a local bound satisfying the condition
\begin{equation}\label{equlocalbound2}
\forall \lambda\in(0,1), \qquad \forall k\in \setK, \qquad\forall P \in \cP,
 \qquad \P_{X\sim P}\Big(|R_k\cap \cH_0(P)| \leq L_k(\lambda) \Big)\leq \lambda\:.
\end{equation}
The properties of the final post hoc bound will obviously depend on the choice of $L_k$.
For instance, the validity of our post hoc bounds relies on \eqref{eq:indep}, which is a strong assumption. The latter is only used to make the DKW inequality valid.
If this assumption is violated, we should use another local bound $L_k$, that satisfies condition \eqref{equlocalbound2} under the specific  dependence setting of the data. For instance, when the dependence is known or satisfies a randomization hypothesis (see \citealp{hemerik2018false}), such a local bound can be easily constructed by applying the $\lambda$-calibration methodology of BNR (e.g., the one corresponding to the balanced template therein).
However, 
the computational complexity of the final post hoc bound will substantially increase, which will make such an approach difficult to use in practice. Solving this problem seems challenging and is left for future work.

\section{Proofs}\label{sec:proof}

\subsection{Proof of Lemma~\ref{lem:Vtilde}}\label{proof:firstlemma}
The second and third inequalities in \eqref{equ:Vtildeconservative}
are straightforward from the fact that $\Vtildeq$ is non-increasing in $q$
and $\Vtilde^1=\Vbar$. For the first inequality,
let $S\subset \Nm$ and consider $A \subset \Nm$ such that $ \forall k  \in \setK$,
$|R_k\cap A| \leq \zeta_k $. For any $Q\subset \setK$, we get
\begin{align*}
|S\cap A|&\leq \sum_{k\in Q} \left| S\cap A\cap R_{k}   \right| + \bigg| S\cap A\cap \comp{\bigg( \bigcup_{k\in Q}R_{k}    \bigg)}  \bigg|\\
&\leq \sum_{k\in Q} \zeta_{k}\wedge \left|S\cap R_{k}   \right|  +  \bigg|S\setminus\bigcup_{k\in Q} R_{k}\bigg|,
\end{align*}
which implies the result.

\subsection{Proof of Theorem~\ref{thm:Forest}}\label{sec:proof:thm}

In this proof, we fix $S\subset\N{m}$. Also, we let
\begin{equation}
\mathcal{A}(\Rfam)=  \left\{ A \subset \Nm : \forall k  \in \setK ,\: |R_k\cap A| \leq \zeta_k\right\},
\label{eq:mathcalA}
\end{equation}
so that $V^*_{\Rfam}(S)=\max_{A\in \mathcal{A}(\Rfam)} |S\cap A|$.
Also note that~\eqref{eq:Vtildeq}--\eqref{eq:VtildeK} can be rewritten as
\begin{equation}
\Vtilde(S)=\min_{\setK' \subset\setK}   \left( \sum_{k\in\setK'} \zeta_{k}\wedge|S\cap R_{k}| +\left|S\setminus\bigcup_{ k\in\setK' } R_{k}\right|  \right).\label{eq:reformuler}
\end{equation}

\subsubsection{Proof of \eqref{itemi}}

First, by Lemma~\ref{lm:completed}, it is sufficient to prove \eqref{itemi} for $\Rfam^\oplus$. Hence, we can focus without generality on the case where \eqref{eq:atom} holds.
%
%
Recall that this means that $(i,i)\in\setK$ for all $1\leq i \leq N$. 
Now, to prove that $\Vtilde(S)=V^*_{\Rfam}(S)$, it suffices to build $A\subset S$ such that $A\in\mathcal{A}\left(\Rfam\right)$ and $|A|= \Vtilde(S)$. 
The key point is that for any $h$, $A$ is the disjoint union of the $A\cap R_{k}$, $k\in\setK^h$, because the $R_{k}$, $k\in\setK^h$, form a partition of $\Nm$ (by Lemma~\ref{lm:part:atom}). 
Let 
$H=\max_{ k\in\setK } \h(k)$
be the greater depth of the Forest structure, we will construct $A$ 
with a decreasing recursion over $h\in\{1,\dotsc,H\}$. To this end, we need some additional notation: first, for any $k\in\setK$, let $\setK_{k}=\{ \tk\in\setK: R_{\tk}\subset R_{k} \}$ be the set of indexes of elements that are subsets of $R_{k}$. Then, for any $h$, let $\setK^{\geq h}=\bigcup_{ h\leq h'\leq H  }\setK^{h'}$. Note that $\setK^{\geq 1}=\setK$. Finally let
$$\mathfrak{P}^h=\{ \mathcal{P}\subset \setK^{\geq h} :\text{ the } R_{k},\, k\in\mathcal{P},\text{ form a partition of } \Nm  \}  ,$$
and note that the result of Lemma~\ref{lm:min:part:atom} (that is, equation~\eqref{eq:minpartition}) can be rewritten in
\begin{equation}
\Vtilde(S)=\min_{ \mathcal{P} \in  \mathfrak{P}^1}  \sum_{k\in\mathcal{P}} \zeta_{k}\wedge|S\cap R_{k}| .
\label{eq:FINALFORM}
\end{equation}

The decreasing recursion starts like this: noting that $\setK^{H}$ is the set of all the $(i,i)$'s, $1\leq i\leq N$, we define $A^{H}$ by choosing (arbitrarily) $\zeta_{i,i}\wedge|S\cap \Rf{i}{i}|$ distinct elements of $S\cap \Rf{i}{i}$ for each $1\leq i\leq N$. Note that we have both
$$\forall k\in\setK^{\geq H}   ,\;\; |A^H\cap R_{k} | \leq \zeta_{k}       ,$$
and
\begin{equation*}
|A^H|=\sum_{k\in\setK^H}  \zeta_{k}\wedge|S\cap R_{k}|
=  \min_{\mathcal{P}\in \mathfrak{P}^{H}}   \sum_{k\in\mathcal{P}}\zeta_{k}\wedge |S\cap R_{k}|  ,
\end{equation*}
since $\mathfrak{P}^{H}=\left\{  \setK^H\right\}$.

Now let $h$ be given and assume we have constructed an $A^{h+1}\subset S$ 
such that both
$$\forall k\in\setK^{\geq h+1}   ,\;\; |A^{h+1}\cap R_{k} | \leq \zeta_{k}       ,$$
and
\begin{align}
|A^{h+1}|&=\min_{\mathcal{P}\in \mathfrak{P}^{h+1}} \sum_{k\in\mathcal{P}}\zeta_{k}\wedge |S\cap R_{k}|\notag \\
&= \sum_{k\in\mathcal{P}^{h+1}}\zeta_{k}\wedge |S\cap R_{k}| , \label{eq:minP}
\end{align}
for a given $\mathcal{P}^{h+1}\in\mathfrak{P}^{h+1}$. Using that $|A^{h+1}|=\sum_{k\in\mathcal{P}^{h+1}}|A^{h+1} \cap R_{k}|$ and that $|A^{h+1}\cap R_{k} | \leq \zeta_{k} \wedge |S\cap R_{k}|  $ for all $k\in\mathcal{P}^{h+1}$, we deduce that $|A^{h+1}\cap R_{k} | = \zeta_{k} \wedge |S\cap R_{k}|  $ for all $k\in\mathcal{P}^{h+1}$.

Now we want to construct $A^h$ by defining all the $A^h\cap R_{k}$, $k\in\setK^h$. By writing that $R_{k}=\bigcup_{\tk \in \mathcal{P}^{h+1}\cap \setK_{k}} R_{\tk}$, the union being disjoint, we have first that,  for all $k\in\setK^h$,
\begin{align*}
 |A^{h+1}\cap R_{k} | &= \sum_{\tk\in \mathcal{P}^{h+1}\cap \setK_{k}  }  |A^{h+1}\cap R_{\tk} | \\
 &= \sum_{\tk\in \mathcal{P}^{h+1}\cap \setK_{k}  }  \zeta_{\tk} \wedge |S\cap R_{\tk}|.
\end{align*}
Second, we have that:
\begin{align}
\min_{\mathcal{P}\in \mathfrak{P}^{h}} \sum_{k\in\mathcal{P}}\zeta_{k}\wedge |S\cap R_{k}|&=\sum_{k\in\setK^h}\min_{\mathcal{P}\in\mathfrak{P}^h}\left( \sum_{\tk\in\mathcal{P}\cap\setK_{k}}\zeta_{\tk} \wedge |S\cap R_{\tk}|\right) \label{eq:chaud1}\\
&=\sum_{k\in\setK^h} \min\Bigg( \zeta_{k}\wedge |S\cap R_{k}| , 
 \min_{\mathcal{P}\in\mathfrak{P}^{h+1}}\left( \sum_{\tk\in\mathcal{P}\cap\setK_{k}}\zeta_{\tk} \wedge |S\cap R_{\tk}|\right)    \Bigg) \label{eq:chaud2} \\
&=\sum_{k\in\setK^h} \min\Bigg( \zeta_{k}\wedge |S\cap R_{k}| , 
\sum_{   \tk\in \mathcal{P}^{h+1}\cap \setK_{k}  }  \zeta_{\tk} \wedge |S\cap R_{\tk}|  \Bigg)  \label{eq:chaud3}\\
&=\sum_{k\in\setK^h} \min\left( \zeta_{k}\wedge |S\cap R_{k}| ,  |A^{h+1}\cap R_{k} |   \right).\notag
\end{align}
In the above, \eqref{eq:chaud1} holds by additivity and because for every $\mathcal{P}\in \mathfrak{P}^{h}$, any element of $\mathcal{P}$ is also an element of one of the $\mathcal{P}\cap\setK_{k}$, $k\in\setK^h$. Moreover, for every $\mathcal{P}\in \mathfrak{P}^{h}$ and $k\in\setK^h$, $\mathcal{P}\cap\setK_{k}$ is either $\{k\}$, either a set of elements of depth $\geq h+1$, hence~\eqref{eq:chaud2}. Finally, \eqref{eq:chaud3} holds because all the minima in~\eqref{eq:chaud2} are realized in $\mathcal{P}^{h+1}$, otherwise the minimality of $\mathcal{P}^{h+1}$ in~\eqref{eq:minP} would be contradicted.

We finally construct all the $A^h\cap R_{k}$, $k\in\setK^h$, in the following way: if $ |A^{h+1}\cap R_{k} |  \leq  \zeta_{k}\wedge |S\cap R_{k}| $, we let $A^h\cap R_{k}=A^{h+1}\cap R_{k} $, else we let $A^h\cap R_{k}$ be a subset of $\zeta_{k}\wedge |S\cap R_{k}| $ distinct elements of $A^{h+1}\cap R_{k} $.
 This both ensures that
$$|A^{h}|=\min_{\mathcal{P}\in \mathfrak{P}^{h}} \sum_{k\in\mathcal{P}}\zeta_{k}\wedge |S\cap R_{k}|  , $$
and that
$$\forall k\in\setK^{\geq h}   ,\;\; |A^h\cap R_{k} | \leq \zeta_{k}       ,$$
because $\setK^{\geq h} =\setK^{ h} \cup \setK^{\geq h+1}  $ and $A^h\subset A^{h+1}$, which ends the recursion.

Now letting $A=A^1$, we have found an $A\subset S$ such that $A\in\mathcal{A}\left(\Rfam\right)$ and $|A|= \Vtilde(S)$ (by~\eqref{eq:FINALFORM}).

\subsubsection{Proof of \eqref{itemii}}

%

By~\eqref{itemi} and Lemmas~\ref{lm:min:part:atom} and~\ref{lm:completed}, we have 
$$V^*_{\Rfam}(S)=V^*_{\Rfam^\oplus}(S) =\widetilde V_{\Rfam^\oplus}(S)=  \sum_{k\in\overline{\setK}} \zeta_{k}\wedge|S\cap R_{k}|, $$
for some $\overline{\setK} \subset \setK^\oplus$ such that the $R_{k}$, $k\in\overline{\setK}$, form a partition of $\N{m}$. Hence,
\begin{align*}
V^*_{\Rfam}(S)&= \sum_{k\in\setK\cap\overline{\setK}} \zeta_{k}\wedge|S\cap R_{k}|+\sum_{k\in\overline{\setK}\setminus\setK} \zeta_{k}\wedge|S\cap R_{k}|\\
&=\sum_{k\in\setK\cap\overline{\setK}} \zeta_{k}\wedge|S\cap R_{k}|+\sum_{k\in\overline{\setK}\setminus\setK}|S\cap R_{k}|\\
&=\sum_{k\in\setK\cap\overline{\setK}} \zeta_{k}\wedge|S\cap R_{k}|+ \left|S\setminus \bigcup_{k\in\setK\cap\overline{\setK}} R_k  \right|,
\end{align*}
because the $R_k$, $k\in\overline{\setK}\setminus\setK $ are all disjoint. Now,  $| \setK\cap\overline{\setK} |\leq d$ by definition of $d$, which means that the latter display is larger than or equal to $\Vtilde^d(S)$, which proves the result.

\subsection{Proof of Corollary~\ref{cor:disjoint}}\label{sec:proof:cor}

\paragraph{Proof of \emph{(i)} }This is a direct byproduct of Theorem~\ref{thm:Forest}, because if \eqref{equ:nested} holds, then $d=1$ and thus  $V^*_{\Rfam}= \Vtilde^d=\Vtilde^1=\Vbar$.

\paragraph{Proof of \emph{(ii)} }By Theorem~\ref{thm:Forest}, $V^*_{\Rfam}=\Vtilde=\Vtilde^K$ defined by~\eqref{eq:Vtildeq}--\eqref{eq:VtildeK}. Now, for any $S\subset \Nm$, for any $Q\subset \setK$ with $|Q|\leq K-1$, by denoting $k_0$ any element not in $Q$, we have
$$
R_{k_0}\cap \left(\bigcup_{k\in Q} R_{k}\right)=\varnothing,
$$
by~\eqref{equ:disjoint}, and
\begin{align*}
 \sum_{k\in Q} \zeta_{k}\wedge|S\cap R_{k}| &+\left|S\setminus\bigcup_{k\in Q} R_{k}\right| =  |S\cap R_{k_0}| +   \sum_{k\in Q} \zeta_{k}\wedge|S\cap R_{k}| +\left|S\setminus\left(\bigcup_{k\in Q} R_{k} \cup R_{k_0}\right)\right| \\
 &\geq  \zeta_{k_0}\wedge |S\cap R_{k_0}| +   \sum_{k\in Q} \zeta_{k}\wedge|S\cap R_{k}| +\left|S\setminus\left(\bigcup_{k\in Q} R_{k} \cup R_{k_0}\right)\right| \\
 &=\sum_{k\in Q\cup\{k_0\}} \zeta_{k}\wedge|S\cap R_{k}| +\left|S\setminus\bigcup_{k\in Q\cup\{k_0\}} R_{k}\right|.
\end{align*}
Hence, the minimum in  \eqref{eq:Vtildeq} within the $\Vtilde^K$  expression is attained for $Q=\mathcal{K}$ and the result is proved.

\subsection{Proof of Proposition~\ref{prop:calibration}}
\label{sub:proof:calibration}
Let us show that for all $\lambda\in (0, 1/2)$, for any $S \subset \Nm$ with cardinal $s=|S|$, we have with probability at least $1-\lambda$ that
\begin{equation}\label{equlocalbound}
|S \cap \cH_0|\leq \min_{t\in[0,1)} \left( \frac{\sqrt{\log(1/\lambda)/2}}{2(1-t)} +
\left\{\frac{\log(1/\lambda)/2}{4(1-t)^2} + \frac{N_t(S)}{1-t}\right\}^{1/2} \right)^2,
\end{equation}
for $N_t(S)=\sum_{i\in S} \mathbf{1}{\{p_i(X) > t\}}$.
Let $v=|S \cap \cH_0|$ (assumed to be positive without loss of generality) and $U_1,\dots, U_{v}$ being $v$ i.i.d. uniform random variables. 
The DKW inequality (with the optimal constant of \citealp{massart1990tight}) ensures that, with probability at least $1-\lambda$, for all $t\in[0,1]$, we have
\begin{align*}
v^{-1} \sum_{i=1}^v \mathbf{1}{\{U_i > t\}} - (1-t) \geq -\sqrt{\log (1/\lambda)/(2v)}.
\end{align*}
Now using Lemma~\ref{lem:trivial} with $x=v^{1/2}$, $a=1-t$, $b= \sqrt{\log(1/\lambda)/2}$ and $c=\sum_{i=1}^v \mathbf{1}{\{U_i > t\}}$ provides~\eqref{equlocalbound} but with $N_t(S)$ replaced by $c$. Since $p_i(X)$ stochastically dominates $U_i$, by independence $N_t(S)$ also dominates $c$, which yields
$$\forall k\in\setK,\: \P\left(\left|R_k \cap \mathcal{H}_0 \right| > \zeta_k(X)  \right) \leq \frac{\alpha}{K} ,  $$
by choosing $\lambda=\frac{\alpha}{K}$. Then~\eqref{eq:jercontrol} follows by a classical union bound argument. 


\subsection{Proof of Proposition~\ref{prop:comparison} }
\label{sub:proof:comparison}
We have for any $t\in[0,1)$,
\begin{align*}
\frac{\E(\VBonf(R_1))}{|R_1|} &= s^{-1} \sum_{i\in R_1\cap \cH_0} \P(p_i(X)>\alpha/m) + s^{-1}\sum_{i\in R_1\cap \cH_1} \P(p_i(X)>\alpha/m)
\\
&=  (1-r)(1-\alpha/m) + r \left(1-\ol{\Phi}(\ol{\Phi}^{-1}(\alpha/m)-\mu)\right),
\end{align*}
which gives \eqref{boundVBonf}.
Next,
\begin{align*}
\VSimes(R_1)&= \min_{1\leq k \leq s}\left\{\sum_{i\in R_1}\ind{p_i(X)>\alpha k/m} + k-1\right\}\\
&\geq \sum_{i\in R_1}\ind{p_i(X)>\alpha s /m},
\end{align*}
which gives \eqref{boundSimes}.
Finally, for all $t\in[0,1)$, by denoting $N=\sum_{i\in R_1} \mathbf{1}{\{p_i(X) > t\}}$, we have
\begin{align*}
\E(\VDKW(R_1)) & \leq  \E  \left[\left( \frac{C}{2(1-t)} +
\left\{\frac{C^2}{4(1-t)^2} + \frac{N}{1-t}\right\}^{1/2} \right)^2\right]\\
&\leq  \E  \left[\left( \frac{C}{1-t} +   \left\{\frac{N}{1-t}\right\}^{1/2} \right)^2\right]\\
&\leq  \frac{C^2}{(1-t)^2} +\frac{\E N}{1-t}+\frac{2C}{(1-t)^{3/2}} \:\E  \left[N^{1/2}\right]\\
&\leq  \frac{C^2}{(1-t)^2} +\frac{\E N}{1-t}+\frac{2C}{1-t} \:\left(\frac{\E  N}{1-t}\right)^{1/2},
\end{align*}
where we used $\sqrt{x+y}\leq \sqrt{x}+\sqrt{y}$ for all $x,y\geq 0$ and that $x\mapsto x^{1/2}$ is concave.
Since
$$
\Esp{  N / |R_1|} = (1-r) (1-t) + r \left(1-\ol{\Phi}(\ol{\Phi}^{-1}(t)-\mu)\right),
$$
and $\Esp{  N} \leq s (1-t)$, this provides
\begin{align*}
\frac{\E(\VDKW(R_1))}{|R_1|}
&\leq \min_t \left\{s^{-1}\frac{C^2}{(1-t)^2} +1-r + r \frac{\ol{\Phi}(\mu-\ol{\Phi}^{-1}(t))}{1-t}+s^{-1/2}\frac{2C}{1-t}\right\}.
\end{align*}
Taking $t=1/2$ 
gives \eqref{boundVDKW}.


\section*{Acknowledgements}

This work has been supported by ANR-16-CE40-0019 (SansSouci) and ANR-17-CE40-0001 (BASICS).

\bibliographystyle{apalike}
\bibliography{biblio}

\appendix

\section{Auxiliary lemmas}\label{app:sec:aux}

The following lemma holds.
\begin{lemma}\label{lem:trivial}
For all $a > 0$ and $b,c,x\geq0$, the two following assertions are equivalent
\begin{itemize}
\item[(i)] $c - a x^2 \geq - bx$;
\item[(ii)] $x\leq \frac{b}{2a} + \sqrt{\frac{b^2}{4a^2} + \frac{c}{a}}$.
\end{itemize}
In particular, we have for all $d\geq 0$,
\begin{equation}\label{equ:majdwedge}
d\wedge \left(\frac{b}{2a} + \sqrt{\frac{b^2}{4a^2} + \frac{c}{a}}\right)^2 \leq d\wedge \left(\frac{c+d^{1/2} b}{a}\right).
\end{equation}
\end{lemma}

\begin{proof}
The equivalence between (i) and (ii) is obvious.
For $d\geq 0$, if we have the inequality $\left(b/(2a) + \sqrt{b^2/(4a^2) + c/a}\right)^2\geq d$, then (ii) is satisfied with $x=d^{1/2}$, which entails
$c - a d\geq - bd^{1/2} $ and gives $d\leq (c +d^{1/2} b)/a$. If, on the contrary, $\left(b/(2a) + \sqrt{b^2/(4a^2) + c/a}\right)^2\leq d$, then
\begin{align*}
 \left(\frac{b}{2a} + \sqrt{\frac{b^2}{4a^2} + \frac{c}{a}}\right)^2&=\frac{b^2}{2a^2}+\frac{c}{a}+\frac{b}{a}\sqrt{b^2/(4a^2) + c/a}\\
 &=\frac{c}{a} +\frac{b}{a} \left(b/(2a) + \sqrt{b^2/(4a^2) + c/a}\right)\leq \frac{c}{a} +\frac{b}{a}d^{1/2}.
\end{align*}
This entails the result.
\end{proof}

The two following lemmas 
are used in the proof of Theorem~\ref{thm:Forest}, in the case where condition~\eqref{eq:atom} holds.
\begin{lemma}
 For a reference family that has a Forest structure, if~\eqref{eq:atom} holds, then for any $h\geq1$, the $\Rf{i}{j}$, $(i,j)\in\setK^h$, form a partition of  $\N{m}$.
 \label{lm:part:atom}
 \end{lemma}
 \begin{proof}
 Let $h\geq1$. Let $(i,j), (\ti,\tj)\in \setK^h$ such that $(i,j)\neq (\ti,\tj)$. By~\eqref{eq:Forest}, either $\Rf{i}{j}$ and $\Rf{\ti}{\tj}$ are disjoint, or, without loss of generality, $\Rf{i}{j}\subset \Rf{\ti}{\tj}$. If $\h(\ti,\tj)=h$ then the latter is not possible because that would mean that $\h(i,j)\geq h+1$. If $\ti=\tj$, then $\Rf{i}{j}\subset \Rf{\ti}{\tj}$ would imply that $P_i\cup\dotsb\cup P_j \subset P_\ti$ which in turn implies $i=j=\ti=\tj$ which is also impossible. So $\Rf{i}{j}$ and $\Rf{\ti}{\tj}$ are disjoint.

 Now take any $e\in\Nm$. $(P_n)_{1\leq n \leq N}$ is a partition so there exists some $n\leq N$ such that $e\in P_n$. If $\h(n,n)\leq h$ then $(n,n)\in\setK^h$. If $\h(n,n) > h$, then $\{k \in\setK : P_n\subsetneq R_k\}$ has at least $h$ elements. Furthermore those elements are nested by~\eqref{eq:Forest}, so there exists $k\in \setK$ such that $P_n\subsetneq R_k$ and $\phi(k)=h$, hence $e\in R_k$ with $k\in\setK^h$. Finally in both cases $e\in \bigcup_{k\in\setK^h}R_k$ so $\Nm=\bigcup_{k\in\setK^h}R_k$, which concludes the proof.
 \end{proof}
\begin{lemma}
 For a reference family that satisfies \eqref{eq:Forest} and \eqref{eq:atom}, we have
 \begin{equation}
\Vtilde(S)=\min_{  \substack{ \overline{\setK}\subset{\setK} \\  \text{the } R_{k}, k\in \overline{\setK}, \\ \text{ form a partition of } \N{m}  }    }  \sum_{k\in\overline{\setK}} \zeta_{k}\wedge|S\cap R_{k}| .
\label{eq:minpartition}
\end{equation}
that is, the minimum in~\eqref{eq:reformuler} is always achieved on a partition of $\N{m}$.
 \label{lm:min:part:atom}
 \end{lemma}
 \begin{proof}
 Let any $\setK'\subset{\setK}$. Because property~\eqref{eq:Forest} is true, there exists $\setK'_1\subset \setK'$ such that the $R_{k}$, $k\in\setK'_1$, are pairwise disjoint, and
$$\forall k\in \setK', \exists  \tk\in \setK'_1, R_{k}\subset R_{\tk}.$$
Note that this implies that $\bigcup_{k\in \setK'_1}R_{k}=\bigcup_{k\in \setK'}R_{k}$. Likewise, because $\setK$ includes all the $(i,i), 1\leq i\leq N$, there exists $\setK'_2\subset \setK$ such that the $R_{k}$, $k\in\setK'_2$, are pairwise disjoint, and
$$\N{m}\setminus \bigcup_{k\in \setK'_1}R_{k} =   \bigcup_{k\in \setK'_2}R_{k}  . $$
Let $\overline{\setK}=\setK'_1\cup\setK'_2$ and note that the $R_{k}$, $k\in \overline{\setK}$, form a partition of $\N{m}$. To conclude the proof of~\eqref{eq:minpartition}, we write that
\begin{equation*}
\begin{split}
& \sum_{k\in\setK'} \zeta_{k}\wedge|S\cap R_{k}| +\left|S\setminus\bigcup_{ k\in\setK' } R_{k}\right|  = \\
& \sum_{k\in\setK'} \zeta_{k}\wedge|S\cap R_{k}| +\left|S\cap\left( \N{m} \setminus \bigcup_{k\in \setK'_1}R_{k}   \right)\right| \geq\\ 
& \sum_{k\in\setK'_1} \zeta_{k}\wedge|S\cap R_{k}| +\sum_{k\in\setK'_2}\left|S\cap R_{k}  \right|  \geq\\
& \sum_{k\in\setK'_1} \zeta_{k}\wedge|S\cap R_{k}| +\sum_{k\in\setK'_2} \zeta_{k}\wedge\left|S\cap R_{k}  \right| =
  \sum_{k\in\overline{\setK}} \zeta_{k}\wedge|S\cap R_{k}| .\qedhere
\end{split}
\end{equation*}
 \end{proof}

The last lemma is useful for the general case where~\eqref{eq:atom} no longer holds, by making use of the completed Forest structure introduced in Definition~\ref{def:complet}.
\begin{lemma}
 For a reference family $\Rfam=(R_k,\zeta_k)_{k\in\setK}$ that has a Forest structure, and $\setK^+$, $\setK^\oplus$, $\Rfam^\oplus$ as in Definition~\ref{def:complet}, we have for all $S\subset\Nm$:
 $$ V^*_{\Rfam^\oplus}(S)=V^*_{\Rfam}(S), $$
 $$ \wt{V}_{\Rfam^\oplus}(S) =\Vtilde(S). $$
 \label{lm:completed}
 \end{lemma}
 \begin{proof}
It is trivial that $\mathcal{A}\left(\Rfam\right)=\mathcal{A}\left(\mathfrak{R^\oplus}\right)$ (see~\eqref{eq:mathcalA}) because $\zeta_{k}=|R_{k}|$ for $k\in\setK^+$, hence $V^*_{\Rfam^\oplus}(S)=V^*_{\Rfam}(S)$. It is also obvious that $\Vtilde(S)\geq \wt{V}_{\Rfam^\oplus}(S) $ by~\eqref{eq:reformuler} and since $\setK\subset\setK^\oplus$.
Now let any $\setK'\subset\setK^\oplus$. Let $\setK'_1=\setK'\cap\setK$ and $\setK'_2=\setK'\cap\setK^+$. Note that $\setK'$ is the disjoint union of $\setK'_1$ and $\setK'_2$. Then,
\begin{align*}
\sum_{k\in\setK'} \zeta_{k}\wedge|S\cap R_{k}| +\left|S\setminus\bigcup_{ k\in\setK' } R_{k}\right|
&= \sum_{k\in\setK'_1} \zeta_{k}\wedge|S\cap R_{k}| + \sum_{k\in\setK'_2} |S\cap R_{k}| +\left|S\setminus\bigcup_{ k\in\setK' } R_{k}\right|  \\ 
&\geq \sum_{k\in\setK'_1} \zeta_{k}\wedge|S\cap R_{k}| + \Bigg|S\setminus   \bigcup_{ k\in\setK'_1 } R_{k}  \Bigg| \\
&\geq \Vtilde(S),
\end{align*}
because $\zeta_{k}=|R_{k}|$ for $k\in\setK'_2$. Hence $ \wt{V}_{\Rfam^\oplus}(S) \geq\Vtilde(S) $, which concludes the proof.
 \end{proof}

\section{Material for Lemma~\ref{lm:partition}}\label{app:sec:leaves}

Algorithm~\ref{algo:partition} below builds $(P_n)$ and follows directly from the proof. It may be useful for the reader to start by looking the algorithm, in order to get a sense of what the formal proof does.

\begin{proof}[Proof of Lemma~\ref{lm:partition}]
Let $H=\max_{k\in\setK}\h(k)$, where $\phi$ is the depth function defined by~\eqref{eq:depth}. We use a recursion to build, for each $1\leq h\leq H$, an integer $N^h\geq1$ and a partition $P^h=(P^h_n)_{1\leq n\leq N^h}$ which satisfy the following three properties:
\begin{align}
    &P^h\text{ is a partition of }\Nm, \tag{$\mathscr{P}_1^h$}\\
    &\forall k \in\setK \text{ such that } \h(k) < h, \exists (i,j)\in\set{1,\dotsc,N^h}^2  : R_k=\bigcup_{i\leq n \leq j} P_n^h, \tag{$\mathscr{P}_2^h$}\\
    &\forall k\in \setK \text{ such that } \h(k) = h, \exists
    n\in\set{1,\dotsc,N^h} : R_k=P_n^h.\tag{$\mathscr{P}_3^h$}
\end{align}
We start the recursion with $h=1$. Let $Succ_{1}=\left\{ k\in\setK : \h(k)=1 \right\}$,
$$New_{1}=\left\{R_k : k\in  Succ_{1} \right\}\cup\left\{\Nm\setminus\bigcup_{k\in  Succ_{1} }R_k  \right\}\setminus\{\varnothing\} , $$
and $N^1=\left|New_{1}\right|$. We let $P^1$ be the family of elements of $New_1$. $(\mathscr{P}_1^1)$ is true because, by~\eqref{eq:Forest}, for $k,\tk\in Succ_1$, $k\neq k'$, $R_k$ and $R_\tk$ are disjoint (otherwise they can't have same depth). $(\mathscr{P}_2^1)$ and $(\mathscr{P}_3^1)$ are trivially true.

Now let $h\in\{2,\dotsc,H\}$ and assume that there exists $N^{h-1}$ and $P^{h-1}$ satisfying $(\mathscr{P}_1^{h-1})$, $(\mathscr{P}_2^{h-1})$ and $(\mathscr{P}_3^{h-1})$. For all $n\in\{1,\dotsc,N^{h-1}\}$, let $$Succ_{h,n}=\left\{ k\in\setK : \h(k)=h \text{ and } R_k\subset P_n^{h-1} \right\},
$$$$ New_{h,n}= \left\{R_k : k\in  Succ_{h,n} \right\}\cup\left\{ P_n^{h-1} \setminus\bigcup_{k\in  Succ_{h,n} }R_k  \right\}\setminus\{\varnothing\} ,$$
$\mathbb{S}^h_n=\sum_{n'=0}^n\left|New_{h,n'}\right|$ (with $\left|New_{h,0}\right|=0$ by convention), and $\left(P^h_{\mathbb{S}^h_{n-1}+1},\dotsc, P^h_{\mathbb{S}^h_n}\right)$ be the family of the elements of $New_{h,n}$. Then let $N^h=\mathbb{S}^h_{N^{h-1}}$ and $P^h=(P^h_1,\dotsc,P^h_{N^h})$. Note that for each $1\leq n\leq N^{h-1}$, $P^{h-1}_n$ is the disjoint union of $P^h_{\mathbb{S}^h_{n-1}+1},\dotsc, P^h_{\mathbb{S}^h_n}$, because by~\eqref{eq:Forest}, for $k,\tk\in Succ_{h,n}$, $k\neq k'$, $R_k$ and $R_\tk$ are disjoint (otherwise they can't have same depth). This and $(\mathscr{P}_1^{h-1})$ imply $(\mathscr{P}_1^{h})$. Let $k\in\setK$ such that $\h(k)<h$, then $(\mathscr{P}_2^{h-1})$ and $(\mathscr{P}_3^{h-1})$ imply that there exists $(i,j)\in\{1,\dotsc,N^{h-1}\}^2$ such that $R_k=\bigcup_{i\leq n \leq j} P_n^{h-1}$. Hence
$$R_k=\bigcup_{ \mathbb{S}^{h-1}_{i-1}+1    \leq n \leq \mathbb{S}^{h-1}_{j}} P_n^{h}   ,$$
and we get $(\mathscr{P}_2^h)$. Finally let $k\in\setK$ such that $\h(k)=h$. Let $\tk$ be the unique element of $\setK$ such that $\h(\tk)=h-1$ and $R_k\subsetneq R_\tk$. By $(\mathscr{P}_3^{h-1})$, there exists $n\in\{1,\dotsc,N^{h-1}\}$ such that $R_\tk=P_n^{h-1}$. Hence $k\in Succ_{h,n}$ and $R_k$ is equal to one of the elements of $New_{h,n}$, which gives us $(\mathscr{P}_3^{h})$.

Now that the recursion has ended, properties $(\mathscr{P}_1^H)$, $(\mathscr{P}_2^H)$ and $(\mathscr{P}_3^H)$ imply the existence of the desired partition. The proof of the converse statement is straightforward from~\eqref{eq:Rij}.
\end{proof}

For the purpose of Algorithm~\ref{algo:partition}, we let $\len$ and $\con$ be the concatenation and length functions such that, for all $S_1,\dotsc,S_n,S_{n+1}\subset \Nm$ and $S=(S_1,\dotsc,S_n)$, $\len(S)=n$, $\con(S,S_{n+1})=(S_1,\dotsc,S_n,S_{n+1})$ if $S_{n+1} \neq \varnothing$ and  $\con(S,\varnothing) = S$.

\begin{algorithm}
\KwData{$\Rfam=(R_{k},\zeta_{k})_{k\in\setK}$ satisfying~\eqref{eq:Forest}.}
\KwResult{$P=(P_n)_{1\leq n \leq N}$ such that for each $k\in\setK$, there exists some $(i,j)$ such that $R_k=\bigcup_{i\leq n\leq j}P_n$.}
$ P \longleftarrow (\Nm)  $\;
$N \longleftarrow 1$\;
$ H \longleftarrow \max_{k\in\setK} \h(k)  $\;
\For{$h\in (1,\dotsc,H) $}{
$newP\longleftarrow ()$\;
	\For{$n \in  \{1,\dotsc,N\}$}{
		$Succ_{h,n} \longleftarrow \{ k \in  \setK : R_{k}\subset P_n, \h(k)=h\}$\;
			\For{$k \in  Succ_{h,n}  $}{
			$newP\longleftarrow \con(newP,R_k)  $\;
			}
		$newP\longleftarrow \con\left(P_n \setminus\bigcup_{ k \in  Succ_{h,n} }R_k,newP\right)  $\;
	}
$P\longleftarrow newP$\;
$N\longleftarrow \len(P)$\;
}
\Return $P$
\caption[Computation of \texorpdfstring{$(P_n)_{1\leq n \leq N}$}{(Pn)}]{Computation of $(P_n)_{1\leq n \leq N}$}
\label{algo:partition}
\end{algorithm}


\checknbdrafts
\end{document}